\documentclass[12pt,a4paper,reqno,twoside]{amsart}

\usepackage[english]{babel}
\usepackage{stmaryrd}
\usepackage{dsfont}
\usepackage[symbol*,ragged]{footmisc}
\usepackage[colorlinks,linkcolor=red,anchorcolor=blue,citecolor=blue,urlcolor=blue]{hyperref}
\usepackage{color,xcolor}

\usepackage{geometry}
\usepackage{amssymb}
\usepackage{amsmath}
\usepackage{mathrsfs}
\usepackage{amsfonts}
\usepackage{epsfig}

\usepackage{amsthm}
\usepackage{amsxtra}
\usepackage{bbding}
\usepackage{epsfig}
\usepackage{graphicx}
\usepackage{latexsym}
\usepackage{mathbbol}
\usepackage{bbold}

\usepackage{pifont}
\usepackage{wasysym}
\usepackage{skull}
\usepackage{float}

\DeclareSymbolFontAlphabet{\mathbb}{AMSb}
\DeclareSymbolFontAlphabet{\mathbbol}{bbold}

\usepackage{amscd}
\usepackage[all]{xy}
\allowdisplaybreaks[4]
\usepackage{setspace}

\theoremstyle{plain}
\newtheorem{theorem}{\normalfont\scshape Theorem}[section]
\newtheorem{proposition}{\normalfont\scshape Proposition}[section]
\newtheorem{lemma}[proposition]{\normalfont\scshape Lemma}
\newtheorem{corollary}[theorem]{\normalfont\scshape Corollary}
\newtheorem*{corollary*}{\normalfont\scshape Corollary}

\theoremstyle{remark}
\newtheorem*{remark*}{\normalfont\scshape Remark}

\numberwithin{equation}{section}
\addtocounter{footnote}{1}

\renewcommand{\footnoterule}{
  \kern -3pt
  \hrule width 2.5in height 0.4pt
  \kern 3pt
}

\makeatletter
\@ifundefined{MakeUppercase}{}{}
\makeatother

\begin{document}
	
\title[ On moments of the error term of the multivariable k-th divisor functions ]
	  { On moments of the error term of the multivariable k-th divisor functions }

\author[Zhen Guo]{Zhen Guo}

\address{Department of Mathematics, China University of Mining and Technology,
         Beijing, 100083, People's Republic of China}

\email{zhen.guo.math@gmail.com}

\author[Xin Li]{Xin Li}

\address{Department of Mathematics, China University of Mining and Technology,
         Beijing, 100083, People's Republic of China}

\email{lixin\_alice@foxmail.com}

\date{}

\footnotetext[1]{Zhen Guo is the corresponding author. 

\quad\,\,
{\textbf{MR(2020) Subject Classification}}: 11N37,11N64

}

\begin{abstract}
Suppose $k\geqslant3$ is an integer. Let $\tau_k(n)$ be the number of ways $n$ can be written as a product of $k$ fixed factors. For any fixed integer $r\geqslant2$, we have the asymptotic formula
\begin{equation*}
 \sum_{n_1,\cdots,n_r\leqslant x}\tau_k(n_1 \cdots n_r)=x^r\sum_{\ell=0}^{r(k-1)}d_{r,k,\ell}(\log x)^{\ell}+O(x^{r-1+\alpha_k+\varepsilon}),
\end{equation*}
where $d_{r,k,\ell}$ and $0<\alpha_k<1$ are computable constants. In this paper we study the mean square of $\Delta_{r,k}(x)$ and give upper bounds for $k\geqslant4$ and an asymptotic formula for the mean square of $\Delta_{r,3}(x)$. We also get an upper bound for the third power moment of $\Delta_{r,3}(x)$. Moreover, we study the first power moment of $\Delta_{r,3}(x)$ and then give a result for the sign changes of it.
\end{abstract}

\maketitle

\section{Introduction and Results}
Let $k\geqslant2$ be an integer, $\tau_k(n)$ denote the number of ways $n$ can be written as a product of $k$ fixed factors. When $k=2$, $\tau_2(n)=\tau(n)$ is the Dirichlet divisor function. The problems about it are important in analytic number theory and hence have a long history \cite{MR0792089,MR0998378,MR2718848}.

For $k\geqslant3$, suppose $x$ is a large real positive number, we have 
\begin{equation}\label{k div aver}
  \sum_{n\leqslant x}\tau_k(n)=xP_{k-1}(\log x)+\Delta_k(x):=M_k(x)+\Delta_k(x),
\end{equation}
where $P_{k-1}(t)$ is a given polynomial in $t$ of degree $k-1$ and $\Delta_k(x)$ is the error term. We denote
\begin{equation}\label{def alpha beta}
  \alpha_k=\inf\left\{a:\Delta_k(x)=O(x^{a})\right\},\qquad\beta_k=\inf\left\{b:\int_{1}^{x}{\Delta_k}^2(t)dt=O(x^{1+2b})\right\}.
\end{equation}

There are many results about the upper bounds and lower bounds for $\alpha_k$ and $\beta_k$. For unified conclusions on $k$, Voronoi \cite{MR1580627} proved that $\alpha_k\leqslant (k-1)/(k+1)$ for $k\geqslant3$ in 1903. In 1916 Hardy \cite{MR1576550} showed that $\alpha_k\geqslant\beta_k\geqslant (k-1)/2k$ holds for $k\geqslant3$. Hardy and Littlewood \cite{MR1575368} proved that $\alpha_k\leqslant (k-1)/(k+2)$ for $k\geqslant4$ in 1923. Ivi\'{c} \cite{MR0792089} gave a summary, namely
\begin{equation}\label{alphak}
\begin{aligned}
  &\alpha_3\leqslant43/96,\qquad\alpha_k\leqslant(3k-4)/4k\qquad(4\leqslant k\leqslant8),\cdots
\end{aligned}
\end{equation}
and
\begin{equation}\label{betak}
\begin{aligned}
  \beta_k=(k-1)/2k, \qquad(k=2,3,4),\qquad \beta_5\leqslant119/260,\qquad \beta_6\leqslant1/2,\qquad \beta_7\leqslant39/70,\cdots
\end{aligned}
\end{equation}

For the results on $k=3$, in 1956, Tong \cite{MR0098718} developed a new method of deriving an asymptotic formula for the mean square for $\Delta_3(x)$, which can be stated as follows:

Suppose $T$ is a large real number, $\varepsilon$ is a sufficiently small real positive number and $\Delta_3(x)$ is defined in (\ref{k div aver}), then
\begin{equation}\label{Tong}
  \int_{1}^{T}{\Delta_3}^2(x)dx=\frac{1}{10\pi^2}\sum_{n=1}^{\infty}\frac{{\tau_3}^2(n)}{n^{4/3}}T^{5/3}+O(T^{5/3-1/14+\varepsilon}).
\end{equation}

Let $r\geqslant2$ be a fixed integer. In this paper we consider the sum
\begin{equation}\label{Delta_r,k def}
  \sum_{n_1,\cdots,n_r\leqslant x}\tau_k(n_1\cdots n_r)=M_{r,k}(x)+\Delta_{r,k}(x),
\end{equation}
where $M_{r,k}(x)$ is the main term and $\Delta_{r,k}(x)$ is the error term. T\'{o}th and Zhai \cite{MR3841555} studied the condition $k=2$. In this paper we first show that

\begin{theorem}\label{U B for M T}
Let $r\geqslant2$, $k\geqslant3$ be fixed integers. Suppose $\alpha_k$ are expressed in (\ref{alphak}), then for any real number $x\geqslant2$, the asymptotic formula
\begin{equation*}
  \sum_{n_1,\cdots,n_r\leqslant x}\tau_k(n_1\cdots n_r)=M_{r,k}(x)+O(x^{r-1+\alpha_k+\varepsilon})
\end{equation*}
holds for any $\varepsilon>0$, and $M_{r,k}(x)$ is expressed by
\begin{equation*}
  M_{r,k}(x)=x^r\sum_{\ell=0}^{r(k-1)}d_{r,k,\ell}(\log x)^{\ell},
\end{equation*}
where $d_{r,k,\ell}$ $(0\leqslant\ell\leqslant r(k-1))$ are computable constants.
\end{theorem}

 The first author \cite{GZ} have studied the mean square for $\Delta_{r,2}(x)$ and have got an asymptotic formula. In this paper we concentrate on the integral
\begin{equation*}
  \int_{1}^{T}{\Delta_{r,k}}^2(x)dx
\end{equation*}
for $k\geqslant3$ and large real $T$. We give an asymptotic formula for $k=3$ and the upper bounds for $k\geqslant4$. The results are stated as follows.
\begin{theorem}\label{mean square 3}
Let $T\geqslant2$ be a large real number and $r\geqslant2$ be a fixed integer. Then the asymptotic formula
\begin{equation*}
  \int_{1}^{T}{\Delta_{r,3}}^2(x)dx=\frac{r^2}{6\pi^2}T^{2r-1/3}L_{4r-4}(\log T)+O(T^{2r-7/18+\varepsilon})
\end{equation*}
holds for any $\varepsilon>0$, where $L_{4r-4}(u)$ is a polynomial in $u$ of degree $(4r-4)$ denoted by
\begin{equation*}
\begin{aligned}
  L_{4r-4}(u)=\sum_{\ell_1,\ell_2=0}^{2(r-1)}D_{r,3,\ell_1,\ell_2}\sum_{t=0}^{\ell_1+\ell_2}\frac{(-1)^t(\ell_1+\ell_2)!}{(2r-\frac{1}{3})^{t+1}(\ell_1+\ell_2-t)!}u^{\ell_1+\ell_2-t},
\end{aligned}
\end{equation*}
and $D_{r,3,\ell_1,\ell_2}$ $(0\leqslant\ell_1,\ell_2\leqslant2(r-1))$ are computable constants.
\end{theorem}

Theorem \ref{mean square 3} can be viewed as an analogue of Tong's result.

For higher power moments of $\Delta_3(x)$, in 1992 Heath-Brown \cite{MR1159354} proved that for a large real positive $T$, the upper bound estimate
\begin{equation}\label{H-B}
  \int_{1}^{T}|{\Delta_3}(x)|^3dx\ll T^{2+\varepsilon}
\end{equation}
holds for any $\varepsilon>0$. This is the best upper bound since the average order of $\Delta_3(x)$ is $1/3$, which can be obtained by (\ref{Tong}). In this paper we give a similar result for $\Delta_{r,3}(x)$.
\begin{theorem}\label{third moment}
  Let $T\geqslant2$ be a large real number, we have
\begin{equation*}
  \int_{1}^{T}|\Delta_{r,3}(x)|^3dx\ll T^{3r-1+\varepsilon}.
\end{equation*}
\end{theorem}
Using Theorem \ref{mean square 3} we have the average order of $\Delta_{r,3}(x)$ is $(r-2/3)$. Thus this is also the best upper bound.

Moreover, we obtain the following result.
\begin{theorem}\label{first moment}
   Let $T\geqslant2$ be a large real number, we have
\begin{equation*}
  \int_{1}^{T}\Delta_{r,3}(x)dx\ll T^{r+1/6+\varepsilon}.
\end{equation*}
\end{theorem}
\begin{corollary}
  For a large real number $T$, $\Delta_{r,3}(x)$ has at least $T^{5/96-\varepsilon}$ sign changes in $[T,2T]$.
\end{corollary}

\begin{theorem}\label{mean square dayu4}
Let $T\geqslant2$ be a large real number, for fixed integer $k\geqslant4$, 
\begin{equation*}
  \int_{1}^{T}{\Delta_{r,k}}^2(x)dx\ll T^{2r-1+2\beta_k+\varepsilon}
\end{equation*}
holds for any $\varepsilon>0$, where $\beta_k$ are expressed by (\ref{betak}).
\end{theorem}

\paragraph{Notation.} Throughout this paper, $\varepsilon$ denotes a sufficiently small real positive number, not necessarily the same at each occurrence. As usual, $\mathbb{R}$ denotes the set of real numbers, $\mathbb{N}$ denotes the set of natural number, $\zeta$ denotes the Riemann-zeta function. For a positive integer $k$, $\tau_k$ denotes the $k$th Dirichlet divisor function. For a complex number $z$, $\Re z$ denotes the real part of $z$. For integers $m$ and $n$, $\binom{m}{n}$ denotes a binomial coefficient ($m>n$). For a complex number $\theta$, the function $e(\theta)$ denotes $e^{2\pi i\theta}$.
\section{Some lemmas}
We present some lemmas.

\begin{lemma}\label{delta32 mean square}
  Let $T\geqslant10$ be a given large real number and $N$ be a real number such that $T^{\varepsilon}\ll N\ll T^{2/3}$, and $\Delta_3(\cdot)$ be defined in (\ref{k div aver}). For any $T\leqslant x\leqslant2T$, define
\begin{equation*}
\begin{aligned}
  \delta_{31}(x,N)&=\frac{x^{1/3}}{\sqrt{3}\pi}\sum_{n\leqslant N}\frac{\tau_3(n)}{n^{2/3}}\cos\left(6\pi(nx)^{1/3}\right),\\
  \delta_{32}(x,N)&=\Delta_3(x)-\delta_{31}(x,N).
\end{aligned}
\end{equation*}
Then we have
\begin{equation*}
  \int_{T}^{2T}{\delta_{32}}^2(x,N)dx\ll T^{5/3+\varepsilon}N^{-1/3}+T^{14/9+\varepsilon}.
\end{equation*}
\end{lemma}
\begin{proof}
  See Lemma 2.12 in Cao, Tanigawa and Zhai \cite{MR4421948}.
\end{proof}

\begin{lemma}\label{polynomial}
  Let $k\geqslant3$, $r\geqslant2$ be integers, and $t_1,\cdots,t_r$ be complex numbers such that $0<|t_1|,\cdots,|t_r|<1$. Then
\begin{equation*}
\begin{aligned}
  &-(r-1)(1-t_1)^k\cdots(1-t_r)^k+\sum_{j=1}^{r}\prod_{\substack{i=1\\i\neq j}}^{r}(1-t_i)^k\\
  =&1-\sum_{j=2}^{r}(j-1)\sum_{1\leqslant a_1<a_2<\cdots<a_j\leqslant r}\sum_{1\leqslant b_1<b_2<\cdots<b_j\leqslant k}
  (-1)^{b_1+\cdots+b_j}\binom{k}{b_1}\cdots\binom{k}{b_j}{t_{a_1}}^{b_1}\cdots{t_{a_j}}^{b_j},
\end{aligned}
\end{equation*}
which means this polynomial does not contain terms of the form $c{t_j}^l$ $(l>0$ and $c\neq0)$ for every $1\leqslant j\leqslant r$.
\end{lemma}
\begin{proof}
  For any given $1\leqslant j\leqslant r$, there are $(r-j)$ terms in the latter series which contains $(1-t_{a_1})^k\cdots(1-t_{a_j})^k$. By comparing the coefficients of $(1-t_{a_1})^k\cdots(1-t_{a_j})^k$ in the first term and in the latter series, and using the binomial theorem we can finish the proof.
\end{proof}

\begin{lemma}\label{Toth1}
Let $k\geqslant2$ be a fixed integer and let $s_1,\cdots,s_k$ be complex numbers. Then for $\Re{s_j}>1 (j=1,2,\cdots,k)$, we have
\begin{equation*}
 \sum_{n_1,\cdots,n_r=1}^{\infty}\frac{\tau_k(n_1\cdots n_r)}{n_1^{s_1}\cdots n_r^{s_r}}=\zeta^k(s_1)\cdots \zeta^k(s_r)F_{r,k}(s_1,\cdots,s_r),
\end{equation*}
where
\begin{equation*}
 F_{r,k}(s_1,\cdots,s_r)=\sum_{n_1,\cdots,n_r=1}^{\infty}\frac{f_{r,k}(n_1,\cdots,n_r)}{n_1^{s_1}\cdots n_r^{s_r}}.
\end{equation*}
This series is absolutely convergent provided that $\Re s_j>0$ and $\Re(s_j+s_l)>1 (1\leqslant j,l\leqslant r)$, and $f_{r,k}(n_1,\cdots,n_r)$ is multiplicative and symmetric in all variables.\\
Moreover, for any function $g(n_1,\cdots,n_r)$ satisfying $g(n_1,\cdots,n_r)\ll\left(\prod_{j=1}^{r}n_j\right)^{\varepsilon}$, the series
 \begin{equation}\label{toth col}
 \sum_{n_1,\cdots,n_r=1}^{\infty}\frac{f_{r,k}(n_1,\cdots,n_r)g(n_1,\cdots,n_r)}{n_1^{s_1}\cdots n_r^{s_r}}
 \end{equation}
is absolutely convergent provided that $\Re s_j>0$ and $\Re(s_j+s_l)>1 (1\leqslant j,l\leqslant r)$.
\end{lemma}

\begin{proof}
The function $n\mapsto\tau_k(n)$ is multiplicative and $\tau_k(p^\nu)=\binom{\nu+k-1}{k-1}$ for every prime power $p^\nu(\nu\geqslant0)$. The function $(n_1,\cdots,n_r)\mapsto\tau_k(n_1\cdots n_r)$ is multiplicative, viewed as a function of $r$ variables. Therefore its multiple Dirichlet series can be expanded into an Euler product. We obtain
\begin{equation*}
\begin{aligned}
  \sum_{n_1,\cdots,n_r=1}^{\infty}\frac{\tau_{k}(n_1\cdots n_r)}{{n_1}^{s_1}\cdots{n_r}^{s_r}}&=\prod_{p}\sum_{\nu_1\cdots\nu_r=0}^{\infty}\frac{\tau_{k}(p^{\nu_1+\cdots+\nu_r})}{p^{\nu_1s_1+\cdots+\nu_rs_r}}\\
  &=\zeta^k(s_1)\cdots\zeta^k(s_r)F_{r,k}(s_1,\cdots,s_r),
\end{aligned}
\end{equation*}
where
\begin{equation*}
\begin{aligned}
  F_{r,k}(s_1,\cdots,s_r)&=\prod_{p}\left(1-\frac{1}{p^{s_1}}\right)^k\cdots\left(1-\frac{1}{p^{s_r}}\right)^k
  \sum_{\nu_1\cdots\nu_r=0}^{\infty}\frac{\tau_k(p^{\nu_1+\cdots+\nu_r})}{p^{\nu_1s_1+\cdots+\nu_rs_r}}\\
  &=\prod_{p}\left(1-\frac{1}{p^{s_1}}\right)^k\cdots\left(1-\frac{1}{p^{s_r}}\right)^k\\
  &\times\left(1-r+\sum_{\nu_1=0}^{\infty}\frac{\tau_k(p^{\nu_1})}{p^{\nu_1s_1}}+\cdots+\sum_{\nu_1=0}^{\infty}\frac{\tau_k(p^{\nu_r})}{p^{\nu_rs_r}}
  +\sum_{\substack{\nu_1\cdots\nu_r=0\\ \#A(\nu_1,\cdots,\nu_r)\geqslant2}}^{\infty}\frac{\tau_k(p^{\nu_1+\cdots+\nu_r})}{p^{v_1s_1+\cdots+v_rs_r}}\right)\\
  &=\prod_{p}\left(1-\frac{1}{p^{s_1}}\right)^k\cdots\left(1-\frac{1}{p^{s_r}}\right)^k\\
  &\times\left(1-r+\left(1-\frac{1}{p^{s_1}}\right)^{-k}+\cdots+\left(1-\frac{1}{p^{s_r}}\right)^{-k}
  +\sum_{\substack{\nu_1\cdots\nu_r=0\\ \#A(\nu_1,\cdots,\nu_r)\geqslant2}}^{\infty}\frac{\tau_k(p^{\nu_1+\cdots+\nu_r})}{p^{v_1s_1+\cdots+v_rs_r}}\right),
\end{aligned}
\end{equation*}
where $\#A(\nu_1,\cdots,\nu_r):=\{j:1\leqslant j\leqslant r, \nu_j\neq0\}$. And the terms $1/p^{\ell s_j}$ (the case $\nu_t=0$ for all $t\neq j$ and $\nu_j=\ell$) only appear in
\begin{equation*}
  \left(1-\frac{1}{p^{s_1}}\right)^k\cdots\left(1-\frac{1}{p^{s_r}}\right)^k\left(1-r+\left(1-\frac{1}{p^{s_1}}\right)^{-k}+\cdots+\left(1-\frac{1}{p^{s_r}}\right)^{-k}\right).
\end{equation*}
Using Lemma \ref{polynomial} in the case that $t_j=1/p^{s_j} (1\leqslant j\leqslant r)$ we obtain that the coefficient of $1/p^{\ell s_j}$ is zero for every $1\leqslant j\leqslant r$ and $\ell>0$.

Hence if $\Re s_j>0$ and $\Re(s_j+s_l)>1 (1\leqslant j,l\leqslant r)$, then $F_{r,k}(s_1,\cdots,s_r)$ is absolutely convergent. And the convergence of (\ref{toth col}) is a direct corollary.
\end{proof}

\begin{lemma}\label{first deri esti}
Suppose $G_0,m_0$ are fixed real positive numbers, let G(x) be a monotonic function defined on $[a,b]$ such that $|G(x)|\leqslant G_0$ and m(x) be a differentiable real function such that $|m^{'}(x)|\geqslant m_0$ on $[a,b]$, $F(\cdot)=\cos(\cdot)$ or $\sin(\cdot)$ or $e(\cdot)$, then
\begin{equation*}
 \int_a^bG(x)F\left(m(x)\right)dx\ll G_0{m_0}^{-1}.
\end{equation*}
\end{lemma}

\begin{proof}
See Lemma 2.1 in Ivi\'{c} \cite{MR0792089}.
\end{proof}

\begin{lemma}\label{T(a,b)}
  Suppose $N$ is a given large real number, $1\leqslant N_1,N_2\leqslant N$ are given real numbers, $a,b$ are integers such that $a,b\ll N$. Define
\begin{equation*}
  T(a,b)=\sum_{\substack{n_1\sim N_1,n_2\sim N_2\\an_1\neq bn_2}}\frac{1}{|an_1-bn_2|}.
\end{equation*}  
Then we have
\begin{equation*}
  T(a,b)\ll(N_1N_2)^{1/2}\log N.
\end{equation*}
\end{lemma}

\begin{proof}
Let $an_1-bn_2=\eta$, we have $\eta\equiv an_1(mod\hspace{0.1cm}b)$, so we can find a constant $c_0$ such that $1\leqslant c_0<b$, $r=bt+c_0$, where $t$ is an integer such that $0<t<2N^2$, thus
\begin{equation*}
\begin{aligned} 
 T(a,b)&=\sum_{n_1\sim N_1}\sum_{\substack{n_2\sim N_2\\an_1\neq bn_2}}\frac{1}{|an_1-bn_2|}\\
 &\leqslant\sum_{n_1\sim N_1}\left(1+\sum_{1\leqslant t\leqslant2N^2}\frac{1}{bt+c_0}\right)\ll N_1\log N,
\end{aligned}
\end{equation*}
similarly we have $T(a,b)\ll N_2\log N$, thus $T(a,b)\ll (N_1N_2)^{1/2}\log N$.
\end{proof}

\begin{lemma}\label{T(x,y)}
Suppose $x,y$ are large real numbers, $r\geqslant2$ is a fixed integer, $s,w$ are given real numbers such that $0<s<1/2<w<1$, $f_{r,3}$ is defined in Lemma \ref{Toth1} in the case k=3. Let $\mathbf{M}_1$ and $\mathbf{M}_2$ denote the vectors $(m_1,\cdots,m_{r})$ and $(m_{k+1},\cdots,m_{2r})$, $D_1$ and $D_2$ denote $\left(\prod_{j=1}^{r-1}m_j\right)$ and $\left(\prod_{j=r+1}^{2r-1}m_j\right)$, respectively. Let
\begin{equation*}
\begin{aligned}
  &T_{g,r,3}(x,y;s,w)=\sum_{\substack{m_1,\cdots,m_{2r}\leqslant x\\n_1,n_2\leqslant y\\\frac{m_r}{m_{2r}}=\frac{n_1}{n_2}}}
 \frac{f_{r,3}(\mathbf{M}_1)f_{r,3}(\mathbf{M}_2)g(\mathbf{M}_1,\mathbf{M}_2)}{D_1D_2(m_rm_{2r})^{s}}
 \cdot\frac{\tau_3(n_1)\tau_3(n_2)}{(n_1n_2)^{w}},\\
 &T_{g,r,3}(s,w)=\sum_{\substack{\mathbf{M}_1,\mathbf{M}_2\in\mathbb{N}^{r}\\n_1,n_2\in\mathbb{N}\\\frac{m_r}{m_{2r}}=\frac{n_1}{n_2}}}
 \frac{f_{r,3}(\mathbf{M}_1)f_{r,3}(\mathbf{M}_2)g(\mathbf{M}_1,\mathbf{M}_2)}{D_1D_2(m_rm_{2r})^{s}}
 \cdot\frac{\tau_3(n_1)\tau_3(n_2)}{(n_1n_2)^{w}},
\end{aligned}
\end{equation*}

where $g(\mathbf{M}_1,\mathbf{M}_2)$ is any function which satisfies $g(\mathbf{M}_1,\mathbf{M}_2)\ll\left(\prod_{j=1}^{2k}m_j\right)^{\varepsilon}$, then we have

(i) $T_{g,r,3}(s,w)$ is absolutely convergent.

(ii) We have
\begin{equation*}
 T_{g,r,3}(s,w)-T_{g,r,3}(x,y;s,w)\ll x^{-2s+\varepsilon}+y^{1-2w+\varepsilon}.
 \end{equation*}
\end{lemma}

\begin{proof}
Use the same argument of Lemma 2.5 in the first author \cite{GZ}.
\end{proof}

\begin{lemma}\label{pan pan}
Suppose $s=\sigma+it$ is a complex number, $f(n)$ is an arithmetic function with Dirichlet series
\begin{equation*}
  F(s)=\sum_{n=1}^{\infty}\frac{f(n)}{n^{s}},
\end{equation*}
which is convergent for $\sigma>1$. Then for any real number $x\geqslant2$ and fixed real number $c>1$, we have
\begin{equation*}
  \sum_{n\leqslant x}f(n)\left(1-\frac{n}{x}\right)=\frac{1}{2\pi i}\int_{(c)}F(s)\frac{x^s}{s(s+1)}ds.
\end{equation*}
\end{lemma}
\begin{proof}
  Using Theorem 5.1(the Perron's formula) in Karatsuba \cite{MR1215269} we can finish the proof.
\end{proof}

\begin{lemma}\label{delta3 first moment}
  Let $T\geqslant2$ be a real number and $\Delta_3(\cdot)$ be defined in (\ref{k div aver}), then we have
\begin{equation*}
  \int_{1}^{T}\Delta_3(x)dx\ll T^{7/6+\varepsilon}.
\end{equation*}
\end{lemma}
\begin{proof}
  Taking $f(n)=\tau_3(n)$ and $x=T$ in Lemma  we have
\begin{equation*}
  \sum_{n\leqslant T}\tau_3(n)\left(T-n\right)=\frac{1}{2\pi i}\int_{(c)}\zeta^3(s)\frac{T^{s+1}}{s(s+1)}ds.
\end{equation*}
Then using (\ref{k div aver}) the left side becomes
\begin{equation*}
  TM_3(T)+T\Delta_3(T)-\sum_{n\leqslant T}n\tau_3(n)=TM_3(T)-\int_{1}^{T}u{M_3}'(u)du+\int_{1}^{T}\Delta_3(u)du,
\end{equation*}
where we use partial summation to get
\begin{equation*}
  \sum_{n\leqslant T}n\tau_3(n)=\int_{1}^{T}u{M_3}'(u)du+T\Delta_3(T)-\int_{1}^{T}\Delta_3(u)du.
\end{equation*}
Thus 
\begin{equation*}
\begin{aligned}
  \int_{1}^{T}\Delta_3(u)du=&\frac{1}{2\pi i}\int_{(c)}\zeta^3(s)\frac{T^{s+1}}{s(s+1)}ds-TM_3(T)+\int_{1}^{T}u{M_3}'(u)du\\
  &=\frac{1}{2\pi i}\int_{(c)}\zeta^3(s)\frac{T^{s+1}}{s(s+1)}ds-\int_{1}^{T}{M_3}(u)du\\
  &=\frac{1}{2\pi i}\int_{(\sigma)}\zeta^3(s)\frac{T^{s+1}}{s(s+1)}ds,
\end{aligned}
\end{equation*}
where $\sigma$ is a real number such that $0<\sigma<1$.

It is well known that $\zeta(s)$ has the functional equation
\begin{equation*}
  \zeta(s)=\chi(s)\zeta(1-s),
\end{equation*}
where $\chi(s)$ satisfies $\chi(s)\ll t^{1/2-\sigma}$ for $s=\sigma+it$ and $0<\sigma<1/2$. Thus for $0<\sigma<1/2$,
\begin{equation*}
\begin{aligned}
  \int_{(\sigma)}\zeta^3(s)\frac{T^{s+1}}{s(s+1)}ds&=\int_{(\sigma)}\chi^3(s)\zeta^3(1-s)\frac{T^{s+1}}{s(s+1)}ds\\
  &\ll T^{\sigma+1}\int_{-i\infty}^{i\infty}t^{-1/2-3\sigma}\zeta^3(1-\sigma+it)dt\\
  &\ll T^{7/6+\varepsilon}\int_{-i\infty}^{i\infty}t^{-1-\varepsilon}\zeta^3(5/6-\varepsilon+it)dt\\
  &\ll T^{7/6+\varepsilon}
\end{aligned} 
\end{equation*}
holds by taking $\sigma=1/6+\varepsilon$, where the convergence of the latter integral can be obtained by integration by parts and the fourth power moment result of $\zeta(s)$:
\begin{equation*}
  \int_{-i\infty}^{i\infty}\zeta^4(5/6-\varepsilon+it)dt\ll1,
\end{equation*}
which can be found in Titchmarsh \cite{MR0046485}. Hence we complete the proof.
\end{proof}

\section{Proof of Theorem \ref{U B for M T} and Expression of \texorpdfstring{${\Delta_{r,k}}(x)$}{}}
\ 
\newline\indent According to Lemma \ref{Toth1},
\begin{equation*}
  \tau_k(n_1\cdots n_r)=\sum_{n_1=m_1d_1,\cdots,n_r=m_rd_r}f_{r,k}(m_1,\cdots,m_r)\tau_k(d_1)\cdots\tau_k(d_r)
\end{equation*}
holds for $n_1,\cdots,n_r\in\mathbb{N}$, where $f_{r,k}$ is defined in Lemma \ref{Toth1}.

Then we deduce by (\ref{k div aver}) that
\begin{equation}\label{dk sum1}
\begin{aligned}
  \sum_{n_1,\cdots,n_r\leqslant x}\tau_k(n_1\cdots n_r)&=\sum_{m_1,\cdots,m_r\leqslant x}f_{r,k}(m_1,\cdots,m_r)\prod_{j=1}^{r}\left(\sum_{d_j\leqslant x/m_j}\tau_k(d_j)\right)\\
  &=\sum_{m_1,\cdots,m_r\leqslant x}f_{r,k}(m_1,\cdots,m_r)\prod_{j=1}^{r}\left(M_k\left(\frac{x}{m_j}\right)+\Delta_k\left(\frac{x}{m_j}\right)\right)\\
  &=\sum_{m_1,\cdots,m_r\leqslant x}f_{r,k}(m_1,\cdots,m_r)\sum_{i=0}^{r}\binom{r}{i}\prod_{j=1}^{r-i}{M_k}\left(\frac{x}{m_j}\right)\prod_{j=1}^{i}{\Delta_k}\left(\frac{x}{m_j}\right).
\end{aligned}
\end{equation}

We evaluate the main term 
\begin{equation*}
\begin{aligned}
M_{r,k}(x):=\sum_{m_1,\cdots,m_r\leqslant x}f_{r,k}(m_1,\cdots,m_r)\prod_{j=1}^{r}{M_k}\left(\frac{x}{m_j}\right).
\end{aligned}
\end{equation*}
Since $M_k(u)=uP_{k-1}(\log u)$ with $P_t(u)$ a polynomial in u of degree $t$, we have
\begin{equation}\label{M prod exp}
\begin{aligned}
\prod_{j=1}^{r}{M_k}\left(\frac{x}{m_j}\right)=\frac{x^r}{m_1\cdots m_r}\sum_{\ell=0}^{r(k-1)}C_{\ell}(\log m_1,\cdots,\log m_r)(\log x)^{\ell},
\end{aligned}
\end{equation}
where
\begin{equation*}
\begin{aligned}
C_{\ell}(\log m_1,\cdots,\log m_r)=\sum_{j_1,\cdots,j_r}c(j_1,\cdots,j_r)(\log m_1)^{j_1}\cdots(\log m_r)^{j_r},
\end{aligned}
\end{equation*}
the sum being over $0<j_t\leqslant k-1(1\leqslant j\leqslant r)$ and $c(j_1,\cdots,j_r)$ are computable constants. Thus we have
\begin{equation}\label{Mrk main}
\begin{aligned}
  M_{r,k}(x)&=x^r\sum_{\ell=0}^{r(k-1)}(\log x)^{\ell}\sum_{m_1,\cdots,m_r\leqslant x}\frac{f_{r,k}(m_1,\cdots,m_r)C_{\ell}(\log m_1,\cdots,\log m_r)}{m_1\cdots m_r}\\
  &=x^r\sum_{\ell=0}^{r(k-1)}d_{r,k,\ell}(\log x)^{\ell}\\
  &\qquad-x^r\sum_{\ell=0}^{r(k-1)}(\log x)^{\ell}{\sum_{m_1,\cdots,m_r}}'\frac{f_{r,k}(m_1,\cdots,m_r)C_{\ell}(\log m_1,\cdots,\log m_r)}{m_1\cdots m_r},
\end{aligned}
\end{equation}
where
\begin{equation*}
\begin{aligned}
d_{r,k,\ell}:=\sum_{m_1,\cdots,m_r=1}^{\infty}\frac{f_{r,k}(m_1,\cdots,m_r)C_{\ell}(\log m_1,\cdots,\log m_r)}{m_1\cdots m_r}
\end{aligned}
\end{equation*}
is convergent by choosing $g(m_1,\cdots,m_r)=C_{\ell}(\log m_1,\cdots,\log m_r)$ in Lemma \ref{Toth1}, and where $\sum'$ means there is at least one $j(1\leqslant j\leqslant r)$ such that $m_j>x$. Without loss of generality, we suppose $m_r>x$.

Since $\log n\ll n^{\varepsilon}$, we obtain
\begin{equation}\label{Mrk error}
\begin{aligned}
  &x^r\sum_{\ell=0}^{r(k-1)}(\log x)^{\ell}{\sum_{m_1,\cdots,m_r}}'\frac{f_{r,k}(m_1,\cdots,m_r)C_{\ell}(\log m_1,\cdots,\log m_r)}{m_1\cdots m_r}\\
  &\ll x^r\sum_{\ell=0}^{r(k-1)}(\log x)^{\ell}\sum_{m_1,\cdots,m_{r-1}=1}^{\infty}\sum_{m_r>x}\frac{|f_{r,k}(m_1,\cdots,m_r)|(m_1\cdots m_r)^{\varepsilon}}{m_1\cdots m_r}\\
  &\ll x^r\sum_{\ell=0}^{r(k-1)}(\log x)^{\ell}
  \sum_{m_1,\cdots,m_{r-1}=1}^{\infty}\sum_{m_r>x}\frac{|f_{r,k}(m_1,\cdots,m_r)|(m_1\cdots m_r)^{\varepsilon}}{m_1\cdots m_{r-1}{m_r}^{3\varepsilon}}\times\frac{1}{{m_r}^{1-3\varepsilon}}\\
  &\ll x^{r+\varepsilon-(1-3\varepsilon)}\sum_{m_1,\cdots,m_r=1}^{\infty}\frac{|f_{r,k}(m_1,\cdots,m_r)|}{(m_1\cdots m_{r-1})^{1-\varepsilon}{m_r}^{2\varepsilon}}\\
  &\ll x^{r-1+\varepsilon},
\end{aligned}
\end{equation}
the convergence of the latter series is obtained by Lemma \ref{Toth1}.

From (\ref{Mrk main}) and (\ref{Mrk error}) we get
\begin{equation}\label{Mrk}
  M_{r,k}(x)=x^r\sum_{\ell=0}^{r(k-1)}d_{r,k,\ell}(\log x)^{\ell}+O(x^{r-1+\varepsilon}).
\end{equation}

By (\ref{k div aver}), (\ref{def alpha beta}) we obtain $M_k(x/m_j)\ll(x/m_j)^{1+\varepsilon}$ and $\Delta_k(x/m_j)\ll(x/m_j)^{\alpha_k+\varepsilon}$ for every $1\leqslant j\leqslant r$, thus the terms in (\ref{dk sum1}) which contains two or more $\Delta_k(x/m_j)$ are
\begin{equation}\label{Deltark error}
\begin{aligned}
  &\ll\sum_{m_1,\cdots,m_r\leqslant x}|f_{r,k}(m_1,\cdots,m_r)|\left(\prod_{j=1}^{r-2}M_k\left(\frac{x}{m_j}\right)\right)
  \left|\Delta_k\left(\frac{x}{m_{r-1}}\right)\right|\left|\Delta_k\left(\frac{x}{m_r}\right)\right|\\
  &\ll x^{r-2+2\alpha_k+\varepsilon}\sum_{m_1,\cdots,m_r\leqslant x}\frac{|f_{r,k}(m_1,\cdots,m_r)|}{m_1\cdots m_{r-2}(m_{r-1}m_r)^{\alpha_k}}\\
  &=x^{r-2+2\alpha_k+\varepsilon}\sum_{m_1,\cdots,m_r\leqslant x}\frac{|f_{r,k}(m_1,\cdots,m_r)|}{m_1\cdots m_{r-2}(m_{r-1}m_r)^{\alpha_k}}
  \times\frac{(m_{r-1}m_r)^{\frac{1}{2}-\alpha_k+\varepsilon}}{(m_{r-1}m_r)^{\frac{1}{2}-\alpha_k+\varepsilon}}\\
  &\ll x^{r-2+2\alpha_k+2(\frac{1}{2}-\alpha_k)+\varepsilon}\sum_{m_1,\cdots,m_r=1}^{\infty}\frac{|f_{r,k}(m_1,\cdots,m_r)|}{m_1\cdots m_{r-2}(m_{r-1}m_r)^{\frac{1}{2}+\varepsilon}}\\
  &\ll x^{r-1+\varepsilon}
\end{aligned}
\end{equation}
by using Lemma \ref{Toth1}.

Above all,
\begin{equation}\label{Deltark exp}
  \Delta_{r,k}(x)={\Delta_{r,k}}^{*}(x)+O(x^{r-1+\varepsilon})
\end{equation}
holds from (\ref{dk sum1}), (\ref{Mrk main}), (\ref{Mrk}) and (\ref{Deltark error}), where
\begin{equation*}
  {\Delta_{r,k}}^{*}(x)=r\sum_{m_1,\cdots,m_r\leqslant x}f_{r,k}(m_1,\cdots,m_r)\prod_{j=1}^{r-1}M_k\left(\frac{x}{m_j}\right)\Delta_k\left(\frac{x}{m_r}\right).
\end{equation*}
Similar to (\ref{Deltark error}) we obtain
\begin{equation*}
\begin{aligned}
  {\Delta_{r,k}}^{*}(x)&\ll\sum_{m_1,\cdots,m_r\leqslant x}|f_{r,k}(m_1,\cdots,m_r)|\left(\prod_{j=1}^{r-1}M_k\left(\frac{x}{m_j}\right)\right)\left|\Delta_k\left(\frac{x}{m_r}\right)\right|\\
  &\ll x^{r-1+\alpha_k+\varepsilon}\sum_{m_1,\cdots,m_r\leqslant x}\frac{|f_{r,k}(m_1,\cdots,m_r)|}{m_1\cdots m_{r-1}{m_r}^{\alpha_k}}\\
  &\ll x^{r-1+\alpha_k+\varepsilon}\sum_{m_1,\cdots,m_r=1}^{\infty}\frac{|f_{r,k}(m_1,\cdots,m_r)|}{m_1\cdots m_{r-1}{m_r}^{\alpha_k}}\\
  &\ll x^{r-1+\alpha_k+\varepsilon},
\end{aligned}
\end{equation*}
by using Lemma \ref{Toth1}. Hence we proof the Theorem \ref{U B for M T}.

\section{Proof of Theorem \ref{mean square dayu4}}

\subsection{Mean Square of \texorpdfstring{${\Delta_{r,k}}^{*}(x)$}{}}
\ 
\newline\indent
Let $T\geqslant2$ be a large real number, and $\mathbf{M}_1$, $\mathbf{M}_2$, $D_1$ and $D_2$ be defined in Lemma \ref{T(x,y)}. 

For $k\geqslant4$, 
\begin{equation*}
\begin{aligned}
  \int_{T}^{2T}({\Delta_{r,k}}^{*}(x))^2dx&\ll\int_{T}^{2T}\sum_{m_1,\cdots,m_{2r}\leqslant2T}|f_{r,k}(\mathbf{M}_1)f_{r,k}(\mathbf{M}_2)|
  \prod_{j=1}^{r-1}M_k\left(\frac{x}{m_j}\right)\prod_{j=r+1}^{2r-1}M_k\left(\frac{x}{m_j}\right)\\
  &\qquad\times\left|\Delta_k\left(\frac{x}{m_r}\right)\Delta_k\left(\frac{x}{m_{2r}}\right)\right|dx.
\end{aligned}
\end{equation*}
Change the order of integration and summation, and use (\ref{k div aver}) we obtain
\begin{equation*}
\begin{aligned}
  \int_{T}^{2T}({\Delta_{r,k}}^{*}(x))^2dx&\ll\sum_{m_1,\cdots,m_{2r}\leqslant2T}|f_{r,k}(\mathbf{M}_1)f_{r,k}(\mathbf{M}_2)|\int_{T}^{2T}
  \prod_{j=1}^{r-1}M_k\left(\frac{x}{m_j}\right)\prod_{j=r+1}^{2r-1}M_k\left(\frac{x}{m_j}\right)\\
  &\qquad\times\left|\Delta_k\left(\frac{x}{m_r}\right)\Delta_k\left(\frac{x}{m_{2r}}\right)\right|dx\\
  &\ll\sum_{m_1,\cdots,m_{2r}\leqslant2T}\frac{|f_{r,k}(\mathbf{M}_1)f_{r,k}(\mathbf{M}_2)|}{D_1D_2}\int_{T}^{2T}x^{2r-2+\varepsilon}
  \left|\Delta_k\left(\frac{x}{m_r}\right)\Delta_k\left(\frac{x}{m_{2r}}\right)\right|dx\\
  &\ll T^{2r-2+\varepsilon}\sum_{m_1,\cdots,m_{2r}\leqslant2T}\frac{|f_{r,k}(\mathbf{M}_1)f_{r,k}(\mathbf{M}_2)|}{D_1D_2}\int_{T}^{2T}
  \left|\Delta_k\left(\frac{x}{m_r}\right)\Delta_k\left(\frac{x}{m_{2r}}\right)\right|dx.
\end{aligned}
\end{equation*}
By the Cauchy-Schwarz's inequality and (\ref{def alpha beta}) we deduce that
\begin{equation*}
\begin{aligned}
  &\int_{T}^{2T}\left|\Delta_k\left(\frac{x}{m_r}\right)\Delta_k\left(\frac{x}{m_{2r}}\right)\right|dx\\
  \ll&\left(\int_{T}^{2T}{\Delta_k}^2\left(\frac{x}{m_r}\right)dx\right)^{1/2}\left(\int_{T}^{2T}{\Delta_k}^2\left(\frac{x}{m_{2r}}\right)dx\right)^{1/2}\\
  \ll&\left(m_r\int_{\frac{T}{m_r}}^{\frac{2T}{m_r}}{\Delta_k}^2\left(u\right)du\right)^{1/2}\left(m_{2r}\int_{\frac{T}{m_{2r}}}^{\frac{2T}{m_{2r}}}{\Delta_k}^2\left(u\right)du\right)^{1/2}\\
  \ll&\frac{T^{1+2\beta_k+\varepsilon}}{{m_r}^{\beta_k}{m_{2r}}^{\beta_k}}.
\end{aligned}
\end{equation*}
Thus
\begin{equation*}
\begin{aligned}
  \int_{T}^{2T}({\Delta_{r,k}}^{*}(x))^2dx&\ll T^{2r-1+2\beta_k+\varepsilon}
  \sum_{m_1,\cdots,m_{2r}\leqslant2T}\frac{|f_{r,k}(\mathbf{M}_1)f_{r,k}(\mathbf{M}_2)|}{D_1{m_r}^{\beta_k}D_2{m_{2r}}^{\beta_k}}\\
  &\ll T^{2r-1+2\beta_k+\varepsilon}\left(\sum_{m_1,\cdots,m_r=1}^{\infty}\frac{|f_{r,k}(\mathbf{M}_1)|}{D_1{m_r}^{\beta_k}}\right)^2\\
  &\ll T^{2r-1+2\beta_k+\varepsilon},
\end{aligned}
\end{equation*}
which means we complete the proof of Theorem in the case of $k\geqslant4$ by replacing $T$ by $T/2$, $T/2^2$, and so on, and adding up all the results.

\section{Proof of Theorem \ref{mean square 3}}

\subsection{Mean Square of \texorpdfstring{${\Delta_{r,3}}^{*}(x)$}{}}
\ 
\newline\indent
In order to evaluate
\begin{equation*}
  \int_{T}^{2T}({\Delta_{r,3}}^{*}(x))^2dx,
\end{equation*}
we divide the ${\Delta_{r,3}}^{*}(x)$ into three parts, namely
\begin{equation*}
  {\Delta_{r,3}}^{*}(x)=M_1+O(M_2+M_3),
\end{equation*}
where
\begin{equation*}
\begin{aligned}
  &M_1=M_1(x,y):=r\sum_{m_1,\cdots,m_r\leqslant y}f_{r,3}(m_1,\cdots,m_r)\prod_{j=1}^{r-1}M_3\left(\frac{x}{m_j}\right)\Delta_3\left(\frac{x}{m_r}\right),\\
  &M_2=M_2(x,y):=\sum_{\substack{m_1,\cdots,m_r\leqslant x\\m_1>y}}|f_{r,3}(m_1,\cdots,m_r)|\prod_{j=1}^{r-1}M_3\left(\frac{x}{m_j}\right)\left|\Delta_3\left(\frac{x}{m_r}\right)\right|,\\
  &M_3=M_3(x,y):=\sum_{\substack{m_1,\cdots,m_r\leqslant x\\m_r>y}}|f_{r,3}(m_1,\cdots,m_r)|\prod_{j=1}^{r-1}M_3\left(\frac{x}{m_j}\right)\left|\Delta_3\left(\frac{x}{m_r}\right)\right|,
\end{aligned}
\end{equation*}
and $y$ is a parameter such that $T^{\varepsilon}\ll y\ll T$. Thus
\begin{equation}\label{Delta* sq exp}
  \int_{T}^{2T}({\Delta_{r,3}}^{*}(x))^2dx=\int_{T}^{2T}{M_1}^2dx+O\left(\int_{T}^{2T}(M_1M_2+M_1M_3)dx\right)+O\left(\int_{T}^{2T}({M_2}^2+{M_3}^2)dx\right).
\end{equation}

Let $\mathbf{M}_1,\mathbf{M}_2,D_1,D_2$ be defined in Lemma \ref{T(x,y)}, then change the order of summation and integration, by (\ref{k div aver}) we obtain
\begin{equation*}
\begin{aligned}
  \int_{T}^{2T}{M_3}^2dx=&\int_{T}^{2T}\sum_{\substack{m_1,\cdots,m_{2r}\leqslant2T\\m_r,m_{2r}>y}}|f_{r,3}(\mathbf{M}_1)f_{r,3}(\mathbf{M}_2)|
  \prod_{j=1}^{r-1}M_3\left(\frac{x}{m_j}\right)\prod_{j=r+1}^{2r-1}M_3\left(\frac{x}{m_j}\right)\\
  &\times\left|\Delta_3\left(\frac{x}{m_r}\right)\Delta_3\left(\frac{x}{m_{2r}}\right)\right|dx\\
  =&\sum_{\substack{m_1,\cdots,m_{2r}\leqslant2T\\m_r,m_{2r}>y}}|f_{r,3}(\mathbf{M}_1)f_{r,3}(\mathbf{M}_2)|
  \int_{T}^{2T}\prod_{j=1}^{r-1}M_3\left(\frac{x}{m_j}\right)\prod_{j=r+1}^{2r-1}M_3\left(\frac{x}{m_j}\right)\\
  &\times\left|\Delta_3\left(\frac{x}{m_r}\right)\Delta_3\left(\frac{x}{m_{2r}}\right)\right|dx\\
  \ll&\sum_{\substack{m_1,\cdots,m_{2r}\leqslant2T\\m_r,m_{2r}>y}}\frac{|f_{r,3}(\mathbf{M}_1)f_{r,3}(\mathbf{M}_2)|}{D_1D_2}
  \int_{T}^{2T}x^{2r-2+\varepsilon}\left|\Delta_3\left(\frac{x}{m_r}\right)\Delta_3\left(\frac{x}{m_{2r}}\right)\right|dx\\
  \ll&T^{2r-2+\varepsilon}\sum_{\substack{m_1,\cdots,m_{2r}\leqslant2T\\m_r,m_{2r}>y}}\frac{|f_{r,3}(\mathbf{M}_1)f_{r,3}(\mathbf{M}_2)|}{D_1D_2}
  \int_{T}^{2T}\left|\Delta_3\left(\frac{x}{m_r}\right)\Delta_3\left(\frac{x}{m_{2r}}\right)\right|dx,
\end{aligned}
\end{equation*}
using Cauchy-Schwarz's inequality and (\ref{def alpha beta}) we deduce that
\begin{equation*}
\begin{aligned}
  \int_{T}^{2T}\left|\Delta_3\left(\frac{x}{m_r}\right)\Delta_3\left(\frac{x}{m_{2r}}\right)\right|dx
  &\ll\left(\int_{T}^{2T}{\Delta_3}^2\left(\frac{x}{m_r}\right)dx\right)^{1/2}\left(\int_{T}^{2T}{\Delta_3}^2\left(\frac{x}{m_{2r}}\right)dx\right)^{1/2}\\
  &\ll\left(m_r\int_{\frac{T}{m_r}}^{\frac{2T}{m_r}}{\Delta_3}^2\left(u\right)du\right)^{1/2}\left(m_{2r}\int_{\frac{T}{m_{2r}}}^{\frac{2T}{m_{2r}}}{\Delta_3}^2\left(u\right)du\right)^{1/2}\\
  &\ll\frac{T^{5/3+\varepsilon}}{{m_r}^{1/3}{m_{2r}}^{1/3}},
\end{aligned}
\end{equation*}
then by Lemma \ref{Toth1} we have
\begin{equation}\label{M_3 mean square}
\begin{aligned}
  \int_{T}^{2T}{M_3}^2dx&\ll T^{2r-1/3+\varepsilon}\left(\sum_{\substack{m_1,\cdots,m_{r}\leqslant2T\\m_r>y}}\frac{|f_{r,3}(\mathbf{M}_1)|}{D_1{m_r}^{1/3}}\right)^2\\
  &\ll T^{2r-1/3+\varepsilon}\left(\sum_{\substack{m_1,\cdots,m_{r}\leqslant2T\\m_r>y}}
  \frac{|f_{r,3}(\mathbf{M}_1)|}{D_1{m_r}^{\varepsilon}}\times\frac{1}{{m_r}^{1/3-\varepsilon}}\right)^2\\
  &\ll T^{2r-1/3+\varepsilon}y^{-2/3+\varepsilon}\left(\sum_{\substack{m_1,\cdots,m_{r}\leqslant2T\\m_r>y}}
  \frac{|f_{r,3}(\mathbf{M}_1)|}{D_1{m_r}^{\varepsilon}}\right)^2\\
  &\ll T^{2r-1/3+\varepsilon}y^{-2/3}.
\end{aligned}
\end{equation}
Let ${D_1}'=D_1/m_1$, similarly we get
\begin{equation}\label{M_2 mean square}
\begin{aligned}
  \int_{T}^{2T}{M_2}^2dx&\ll T^{2r-1/3+\varepsilon}\left(\sum_{\substack{m_1,\cdots,m_{r}\leqslant2T\\m_1>y}}\frac{|f_{r,3}(\mathbf{M}_1)|}{D_1{m_r}^{1/3}}\right)^2\\
  &\ll T^{2r-1/3+\varepsilon}\left(\sum_{\substack{m_1,\cdots,m_{r}\leqslant2T\\m_1>y}}\frac{|f_{r,3}(\mathbf{M}_1)|}{{m_1}^{2/3+\varepsilon}{D_1}'{m_r}^{1/3}}
  \times\frac{1}{{m_1}^{1/3-\varepsilon}}\right)^2\\
  &\ll T^{2r-1/3+\varepsilon}y^{-2/3+\varepsilon}\left(\sum_{\substack{m_1,\cdots,m_{r}\leqslant2T\\m_1>y}}
  \frac{|f_{r,3}(\mathbf{M}_1)|}{{m_1}^{2/3+\varepsilon}{D_1}'{m_r}^{1/3}}\right)^2\\
  &\ll T^{2r-1/3+\varepsilon}y^{-2/3}.
\end{aligned}
\end{equation}
By (\ref{Delta* sq exp}), (\ref{M_3 mean square}), (\ref{M_2 mean square}) we conclude that
\begin{equation*}
  \int_{T}^{2T}({\Delta_{r,k}}^{*}(x))^2dx=\int_{T}^{2T}{M_1}^2dx+O\left(\int_{T}^{2T}(M_1M_2+M_1M_3)dx\right)+O(T^{2r-1/3+\varepsilon}y^{-2/3}).
\end{equation*}

It remains to evaluate $\int_{T}^{2T}{M_1}^2dx$, and then it will follows that $\int_{T}^{2T}M_1M_2dx$ and  $\int_{T}^{2T}M_1M_3dx$ can be estimated by the Cauchy-Schwarz's inequality.

\subsection{Mean Square of \texorpdfstring{$M_1(x,y)$}{}}
Let $N$ be a parameter such that $1\ll N\ll x/m_r\ll x/y$, then
\begin{equation}\label{M1 chaifen}
\begin{aligned}
  M_1(x,y)&=r\sum_{m_1,\cdots,m_r\leqslant y}f_{r,3}(\mathbf{M_1})\prod_{j=1}^{r-1}M_3\left(\frac{x}{m_j}\right)\Delta_3\left(\frac{x}{m_r}\right)\\
  &:=M_{11}(x,y,N)+O(M_{12}(x,y,N)),
\end{aligned}
\end{equation}
where
\begin{equation*}
\begin{aligned}
  &\qquad M_{11}(x,y,N)\\
  &=r\sum_{m_1,\cdots,m_r\leqslant y}f_{r,3}(\mathbf{M_1})\prod_{j=1}^{r-1}M_3\left(\frac{x}{m_j}\right)\frac{x^{1/3}}{\sqrt{3}\pi m_r^{1/3}}\sum_{n\leqslant N}
  \frac{\tau_3(n)}{n^{2/3}}\cos\left(6\pi\left(\frac{nx}{m_r}\right)^{\frac{1}{3}}\right),
\end{aligned}
\end{equation*}
and
\begin{equation*}
\begin{aligned}
M_{12}(x,y,N)=\sum_{m_1,\cdots,m_r\leqslant y}|f_{r,3}(\mathbf{M}_1)|\prod_{j=1}^{r-1}M_3\left(\frac{x}{m_j}\right)\left|\delta_{32}\left(\frac{x}{m_r},N\right)\right|,
\end{aligned}
\end{equation*}
where $\delta_{32}(\cdot,\cdot)$ is defined in Lemma \ref{delta32 mean square}.
Thus
\begin{equation}\label{int M_1 chaifen}
\begin{aligned}
  &\qquad\int_{T}^{2T}{M_1}^2(x,y)dx\\
  &=\int_{T}^{2T}{M_{11}}^2(x,y,N)dx+O\left(\int_{T}^{2T}(M_{11}(x,y,N)M_{12}(x,y,N)+{M_{12}}^2(x,y,N))dx\right).
\end{aligned}
\end{equation}

We are going to evaluate the mean square of $M_{11}$. Using (\ref{M prod exp}) we obtain
\begin{equation*}
\begin{aligned}
  &\qquad M_{11}^2(x,y,N)\\
  &=\frac{r^2x^{2/3}}{3\pi^2}\sum_{m_1,\cdots,m_{2r}\leqslant y}\frac{f_{r,3}(\mathbf{M}_1)f_{r,3}(\mathbf{M}_2)}{(m_rm_{2r})^{1/3}}
  \prod_{j=1}^{r-1}M_3\left(\frac{x}{m_j}\right)\prod_{j=r+1}^{2r-1}M_3\left(\frac{x}{m_j}\right)
  \sum_{n_1,n_2\leqslant N}\frac{\tau_3(n_1)\tau_3(n_2)}{(n_1n_2)^{2/3}}\\
  &\qquad\times\cos\left(6\pi\left(\frac{n_1x}{m_r}\right)^{1/3}\right)\cos\left(6\pi\left(\frac{n_2x}{m_{2r}}\right)^{1/3}\right)\\
  &=\frac{r^2x^{2r-4/3}}{3\pi^2}\sum_{\ell_1,\ell_2=0}^{2(r-1)}(\log x)^{\ell_1+\ell_2}\\
  &\qquad\times\sum_{m_1,\cdots,m_{2r}\leqslant y}\frac{f_{r,3}(\mathbf{M}_1)f_{r,3}(\mathbf{M}_2)
  C_{\ell_1}(\log m_1,\cdots,\log m_{r-1})C_{\ell_2}(\log m_{m_{r+1}},\cdots,\log m_{2r-1})}{D_1D_2(m_rm_{2r})^{1/3}}\\
  &\qquad\times\sum_{n_1,n_2\leqslant N}\frac{\tau_3(n_1)\tau_3(n_2)}{(n_1n_2)^{2/3}}
  \cos\left(6\pi\left(\frac{n_1x}{m_r}\right)^{1/3}\right)\cos\left(6\pi\left(\frac{n_2x}{m_{2r}}\right)^{1/3}\right).
\end{aligned}
\end{equation*}
We denote $C_{\ell_1}(\log m_1,\cdots,\log m_{r-1})C_{\ell_2}(\log m_{m_{r+1}},\cdots,\log m_{2r-1})$ by $C_{\ell_1,\ell_2}(\mathbf{M}_1,\mathbf{M}_2)$, then using $\cos\alpha\cos\beta=\frac{1}{2}(\cos(\alpha-\beta)+\cos(\alpha+\beta))$ we get
\begin{equation}\label{int 012 def}
  M_{11}^2(x,y,N)=S_0(x,y,N)+S_1(x,y,N)+S_2(x,y,N),
\end{equation}
where
\begin{equation*}
\begin{aligned}
  S_0(x,y,N)&=\frac{r^2x^{2r-4/3}}{6\pi^2}\sum_{\ell_1,\ell_2=0}^{2(r-1)}(\log x)^{\ell_1+\ell_2}\\
  &\qquad\times\sum_{\substack{m_1,\cdots,m_{2r}\leqslant y\\n_1,n_2\leqslant N\\ \frac{n_1}{m_r}=\frac{n_2}{m_{2r}}}}
  \frac{f_{r,3}(\mathbf{M}_1)f_{r,3}(\mathbf{M}_2)C_{\ell_1,\ell_2}(\mathbf{M}_1,\mathbf{M}_2)}
  {D_1D_2(m_rm_{2r})^{1/3}}\cdot\frac{\tau_3(n_1)\tau_3(n_2)}{(n_1n_2)^{2/3}},\\
  S_1(x,y,N)&=\frac{r^2x^{2r-4/3}}{6\pi^2}\sum_{\ell_1,\ell_2=0}^{2(r-1)}(\log x)^{\ell_1+\ell_2}\\
  &\qquad\times\sum_{\substack{m_1,\cdots,m_{2r}\leqslant y\\n_1,n_2\leqslant N\\n_1m_{2r}\neq n_2m_r}}
  \frac{f_{r,3}(\mathbf{M}_1)f_{r,3}(\mathbf{M}_2)C_{\ell_1,\ell_2}(\mathbf{M}_1,\mathbf{M}_2)}
  {D_1D_2(m_rm_{2r})^{1/3}}\cdot\frac{\tau_3(n_1)\tau_3(n_2)}{(n_1n_2)^{2/3}}\\
  &\qquad\times\cos\left(6\pi\left(\left(\frac{n_1x}{m_r}\right)^{1/3}-\left(\frac{n_2x}{m_{2r}}\right)^{1/3}\right)\right),
\end{aligned}
\end{equation*}
and
\begin{equation*}
\begin{aligned}
  S_2(x,y,N)&=\frac{r^2x^{2r-4/3}}{6\pi^2}\sum_{\ell_1,\ell_2=0}^{2(r-1)}(\log x)^{\ell_1+\ell_2}\\
  &\qquad\times\sum_{\substack{m_1,\cdots,m_{2r}\leqslant y\\n_1,n_2\leqslant N}}
  \frac{f_{r,3}(\mathbf{M}_1)f_{r,3}(\mathbf{M}_2)C_{\ell_1,\ell_2}(\mathbf{M}_1,\mathbf{M}_2)}
  {D_1D_2(m_rm_{2r})^{1/3}}\cdot\frac{\tau_3(n_1)\tau_3(n_2)}{(n_1n_2)^{2/3}}\\
  &\qquad\times\cos\left(6\pi\left(\left(\frac{n_1x}{m_r}\right)^{1/3}+\left(\frac{n_2x}{m_{2r}}\right)^{1/3}\right)\right).
\end{aligned}
\end{equation*}

Then it turns to deal with
\begin{equation*}
\begin{aligned}
  \int_0=\int_{T}^{2T}S_0(x,y,N)dx,\qquad \int_1=\int_{T}^{2T}S_1(x,y,N)dx,\qquad \int_2=\int_{T}^{2T}S_2(x,y,N)dx.
\end{aligned}
\end{equation*}

\subsubsection{Evaluation of \texorpdfstring{$\int_0$}{}}
\ 
\newline\indent Since $C_{\ell_1,\ell_2}(\mathbf{M}_1,\mathbf{M}_2)$ satisfies
\begin{equation*}
  C_{\ell_1,\ell_2}(\mathbf{M}_1,\mathbf{M}_2)\ll\left(\prod_{j=1}^{2r}m_j\right)^{\varepsilon},
\end{equation*}
we choose $g(\mathbf{M}_1,\mathbf{M}_2))=g_{\ell_1,\ell_2}(\mathbf{M}_1,\mathbf{M}_2)=C_{\ell_1,\ell_2}(\mathbf{M}_1,\mathbf{M}_2)$, $s=1/3$, $w=2/3$ in Lemma \ref{T(x,y)}, then
\begin{equation*}
\begin{aligned}
  S_0(x,y,N)&=\frac{r^2x^{2r-4/3}}{6\pi^2}\sum_{\ell_1,\ell_2=0}^{2(r-1)}(\log x)^{\ell_1+\ell_2}T_{g,r,3}\left(y,N;\frac{1}{3},\frac{2}{3}\right)\\
  &=\frac{r^2x^{2r-4/3}}{6\pi^2}\sum_{\ell_1,\ell_2=0}^{2(r-1)}(\log x)^{\ell_1+\ell_2}\left(T_{g,r,3}\left(\frac{1}{3},\frac{2}{3}\right)+O(y^{-2/3+\varepsilon}+N^{-1/3+\varepsilon})\right).
\end{aligned}
\end{equation*}
Since $g(\mathbf{M}_1,\mathbf{M}_2)$ is related to $\ell_1,\ell_2$, we denote $T_{g,r,3}\left(\frac{1}{3},\frac{2}{3}\right)$ by $D_{r,3,\ell_1,\ell_2}$, therefore
\begin{equation*}
  S_0(x,y,N)=\frac{r^2x^{2r-4/3}}{6\pi^2}Q_{4r-4}(\log x)+O(x^{2r-4/3+\varepsilon}(y^{-2/3}+N^{-1/3})),
\end{equation*}
where 
\begin{equation*}
  Q_{4r-4}(t)=\sum_{\ell_1,\ell_2=0}^{2(r-1)}D_{r,3,\ell_1,\ell_2}t^{\ell_1+\ell_2}
\end{equation*}
is a polynomial of degree $4r-4$. Then it follows that
\begin{equation}\label{int0 eva}
  \int_0=\frac{r^2}{6\pi^2}\sum_{\ell_1,\ell_2=0}^{2(r-1)}D_{r,3,\ell_1,\ell_2}\int_{T}^{2T}x^{2r-4/3}(\log x)^{\ell_1+\ell_2}dx+O(T^{2r-1/3+\varepsilon}(y^{-2/3}+N^{-1/3})).
\end{equation}

\subsubsection{Estimates of \texorpdfstring{$\int_1$}{} and \texorpdfstring{$\int_2$}{}}
\ 
\newline\indent For $\int_2$, change the order of integration and summation we have
\begin{equation*}
\begin{aligned}
  \int_2&=\frac{r^2}{6\pi^2}\sum_{\ell_1,\ell_2=0}^{2(r-1)}\sum_{\substack{m_1,\cdots,m_{2r}\leqslant y\\n_1,n_2\leqslant N}}
  \frac{f_{r,3}(\mathbf{M}_1)f_{r,3}(\mathbf{M}_2)C_{\ell_1,\ell_2}(\mathbf{M}_1,\mathbf{M}_2)}
  {D_1D_2(m_rm_{2r})^{1/3}}\cdot\frac{\tau_3(n_1)\tau_3(n_2)}{(n_1n_2)^{2/3}}\\
  &\qquad\times\int_{T}^{2T}x^{2r-4/3}(\log x)^{\ell_1+\ell_2}\cos\left(6\pi\left(\left(\frac{n_1x}{m_r}\right)^{1/3}+\left(\frac{n_2x}{m_{2r}}\right)^{1/3}\right)\right)dx.
\end{aligned}
\end{equation*}
Choosing $F(\cdot)=\cos(\cdot)$, $G(x)=x^{2r-4/3}(\log x)^{\ell_1+\ell_2}$ and $m(x)=(n_1x/m_r)^{1/3}+(n_2x/m_{2r})^{1/3}$ in Lemma \ref{first deri esti}, then
\begin{equation*}
\begin{aligned}
  \int_2&\ll\sum_{\ell_1,\ell_2=0}^{2(r-1)}\sum_{\substack{m_1,\cdots,m_{2r}\leqslant y\\n_1,n_2\leqslant N}}
  \frac{|f_{r,3}(\mathbf{M}_1)f_{r,3}(\mathbf{M}_2)C_{\ell_1,\ell_2}(\mathbf{M}_1,\mathbf{M}_2)|}{D_1D_2(m_rm_{2r})^{1/3}}\cdot\frac{\tau_3(n_1)\tau_3(n_2)}{(n_1n_2)^{2/3}}\\
  &\qquad\times T^{2r-4/3+\varepsilon}\cdot\frac{T^{2/3}}{(n_1/m_r)^{1/3}+(n_2/m_{2r})^{1/3}}.
\end{aligned}
\end{equation*}

We use $C_{\ell_1,\ell_2}(\mathbf{M}_1,\mathbf{M}_2)\ll\left(\prod_{j=1}^{2r}m_j\right)^{\varepsilon}\ll y^{\varepsilon}\ll T^{\varepsilon}$ and $a^2+b^2\geqslant2ab$ to get
\begin{equation}\label{int2 est}
\begin{aligned}
  \int_2&\ll T^{2r-2/3+\varepsilon}\sum_{\substack{m_1,\cdots,m_{2r}\leqslant y\\n_1,n_2\leqslant N}}
  \frac{|f_{r,3}(\mathbf{M}_1)f_{r,3}(\mathbf{M}_2)|}{D_1D_2(m_rm_{2r})^{1/3}}\cdot\frac{\tau_3(n_1)\tau_3(n_2)}{(n_1n_2)^{2/3}}\left(\frac{m_rm_{2r}}{n_1n_2}\right)^{1/6}\\
  &\ll T^{2r-2/3+\varepsilon}\sum_{n_1,n_2\leqslant N}\frac{\tau_3(n_1)\tau_3(n_2)}{(n_1n_2)^{5/6}}
  \sum_{m_1,\cdots,m_{2r}\leqslant y}\frac{|f_{r,3}(\mathbf{M}_1)f_{r,3}(\mathbf{M}_2)|}{D_1D_2(m_rm_{2r})^{1/6}}\\
  &\ll T^{2r-2/3+\varepsilon}N^{1/3}\left(\sum_{m_1,\cdots,m_r=1}^{\infty}\frac{|f_{r,3}(\mathbf{M}_1)|}{D_1{m_r}^{1/6}}\right)^2\\
  &\ll T^{2r-2/3+\varepsilon}N^{1/3},
\end{aligned}
\end{equation}
where we use partial summation on $n_1,n_2$ and the convergence of the latter series is obtained by Lemma \ref{Toth1}.

Then we turn to estimate $\int_1$, similar to $\int_2$, we have
\begin{equation}\label{def R_1 R_2}
\begin{aligned}
  \int_1&\ll\sum_{\ell_1,\ell_2=0}^{2(r-1)}\sum_{\substack{m_1,\cdots,m_{2r}\leqslant y\\n_1,n_2\leqslant N\\n_1m_{2r}\neq n_2m_r}}
  \frac{|f_{r,3}(\mathbf{M}_1)f_{r,3}(\mathbf{M}_2)C_{\ell_1,\ell_2}(\mathbf{M}_1,\mathbf{M}_2)|}{D_1D_2(m_rm_{2r})^{1/3}}\cdot\frac{\tau_3(n_1)\tau_3(n_2)}{(n_1n_2)^{2/3}}\\
  &\qquad\times T^{2r-4/3+\varepsilon}\cdot\frac{T^{2/3}}{\left|(n_1/m_r)^{1/3}-(n_2/m_{2r})^{1/3}\right|}\\
  &\ll T^{2r-2/3+\varepsilon}\sum_{m_1,\cdots,m_{2r}\leqslant y}\frac{|f_{r,3}(\mathbf{M}_1)f_{r,3}(\mathbf{M}_2)|}{D_1D_2(m_rm_{2r})^{1/3}}
  \sum_{\substack{n_1,n_2\leqslant N\\n_1m_{2r}\neq n_2m_r}}\frac{\tau_3(n_1)\tau_3(n_2)}{(n_1n_2)^{2/3}\left|(n_1/m_r)^{1/3}-(n_2/m_{2r})^{1/3}\right|}\\
  &:=T^{2r-2/3+\varepsilon}\sum_{m_1,\cdots,m_{2r}\leqslant y}\frac{|f_{r,3}(\mathbf{M}_1)f_{r,3}(\mathbf{M}_2)|}{D_1D_2(m_rm_{2r})^{1/3}}(R_1+R_2),
\end{aligned}
\end{equation}
where
\begin{equation*}
\begin{aligned}
  R_1&=\sum_{\substack{n_1,n_2\leqslant N\\n_1m_{2r}\neq n_2m_r\\
  \left|\left(\frac{n_1}{m_r}\right)^{1/3}-\left(\frac{n_2}{m_{2r}}\right)^{1/3}\right|<\frac{1}{10}\left(\frac{n_1n_2}{m_rm_{2r}}\right)^{1/6}}}
  \frac{\tau_3(n_1)\tau_3(n_2)}{(n_1n_2)^{2/3}\left|(n_1/m_r)^{1/3}-(n_2/m_{2r})^{1/3}\right|}\\
  R_2&=\sum_{\substack{n_1,n_2\leqslant N\\n_1m_{2r}\neq n_2m_r\\
  \left|\left(\frac{n_1}{m_r}\right)^{1/3}-\left(\frac{n_2}{m_{2r}}\right)^{1/3}\right|\geqslant\frac{1}{10}\left(\frac{n_1n_2}{m_rm_{2r}}\right)^{1/6}}}
  \frac{\tau_3(n_1)\tau_3(n_2)}{(n_1n_2)^{2/3}\left|(n_1/m_r)^{1/3}-(n_2/m_{2r})^{1/3}\right|}.\\
\end{aligned}
\end{equation*}
Then we have
\begin{equation}\label{R_2}
  R_2\ll\sum_{n_1,n_2\leqslant N}\frac{\tau_3(n_1)\tau_3(n_2)}{(n_1n_2)^{2/3}}\cdot\left(\frac{m_rm_{2r}}{n_1n_2}\right)^{1/6}\ll(m_rm_{2r})^{1/6}N^{1/3+\varepsilon}
\end{equation}
by using partial summation.

For $R_1$, using the Lagrange's Mean Value Theorem, for $r_1,r_2\in\mathbb{R}$, 
\begin{equation*}
|{r_1}^{1/3}-{r_2}^{1/3}|\asymp(r_1r_2)^{-1/3}|r_1-r_2|
\end{equation*}
holds when $r_1\asymp r_2$. Thus choosing $r_1=n_1/m_r$, $r_2=n_2/m_{2r}$ we get
\begin{equation*}
\begin{aligned} 
  R_1&\ll\sum_{\substack{n_1,n_2\leqslant N\\n_1m_{2r}\neq n_2m_r}}\frac{\tau_3(n_1)\tau_3(n_2)}{(n_1n_2)^{2/3}\left|n_1/m_r-n_2/m_{2r}\right|}\cdot\left(\frac{m_rm_{2r}}{n_1n_2}\right)^{-1/3}\\
  &=(m_rm_{2r})^{2/3}\sum_{\substack{n_1,n_2\leqslant N\\n_1m_{2r}\neq n_2m_r}}\frac{\tau_3(n_1)\tau_3(n_2)}{(n_1n_2)^{1/3}\left|n_1m_{2r}-n_2m_r\right|}.
\end{aligned}
\end{equation*}
Let $N_1,N_2$ satisfy $1\leqslant N_1,N_2\leqslant N$. Then using $\tau_3(n)\ll n^{\varepsilon}$ we deduce that
\begin{equation}\label{R_1}
\begin{aligned}
  R_1&\ll(m_rm_{2r})^{2/3}\log^2N\sum_{\substack{n_1\sim N_1,n_2\sim N_2\\n_1m_{2r}\neq n_2m_r}}\frac{\tau_3(n_1)\tau_3(n_2)}{(n_1n_2)^{1/3}\left|n_1m_{2r}-n_2m_r\right|}\\
  &\ll\frac{(m_rm_{2r})^{2/3}}{(N_1N_2)^{1/3}}N^{\varepsilon}\sum_{\substack{n_1\sim N_1,n_2\sim N_2\\n_1m_{2r}\neq n_2m_r}}
  \frac{1}{\left|n_1m_{2r}-n_2m_r\right|}\\
  &\ll\frac{(m_rm_{2r})^{2/3}}{(N_1N_2)^{1/3}}N^{\varepsilon}\cdot(N_1N_2)^{1/2}\log N\\
  &\ll(m_rm_{2r})^{2/3}N^{1/3+\varepsilon},
\end{aligned}
\end{equation}
where we use Lemma \ref{T(a,b)} in the case that $a=m_{2r},b=m_r$.

Therefore we conclude that
\begin{equation*}
  R_1+R_2\ll(m_rm_{2r})^{2/3}N^{1/3+\varepsilon}
\end{equation*}
by (\ref{R_2}) and (\ref{R_1}). Thus we obtain by (\ref{def R_1 R_2}) and Lemma \ref{Toth1} that
\begin{equation}\label{int1 est}
\begin{aligned}
  \int_1&\ll T^{2r-2/3+\varepsilon}N^{1/3}\sum_{m_1,\cdots,m_{2r}\leqslant y}\frac{|f_{r,3}(\mathbf{M}_1)f_{r,3}(\mathbf{M}_2)|}{D_1D_2(m_rm_{2r})^{1/3}}\cdot(m_rm_{2r})^{2/3}\\
  &\ll T^{2r-2/3+\varepsilon}N^{1/3}\sum_{m_1,\cdots,m_{2r}\leqslant y}\frac{|f_{r,3}(\mathbf{M}_1)f_{r,3}(\mathbf{M}_2)|(m_rm_{2r})^{1/3+\varepsilon}}{D_1D_2(m_rm_{2r})^{\varepsilon}}\\
  &\ll T^{2r-2/3+\varepsilon}N^{1/3}y^{2/3}\sum_{m_1,\cdots,m_{2r}\leqslant y}\frac{|f_{r,3}(\mathbf{M}_1)f_{r,3}(\mathbf{M}_2)|}{D_1D_2(m_rm_{2r})^{\varepsilon}}\\
  &\ll T^{2r-2/3+\varepsilon}N^{1/3}y^{2/3}\left(\sum_{m_1,\cdots,m_r=1}^{\infty}\frac{|f_{r,3}(\mathbf{M}_1)|}{D_1{m_r}^{\varepsilon}}\right)^2\\
  &\ll T^{2r-2/3+\varepsilon}N^{1/3}y^{2/3}.
\end{aligned}
\end{equation}

Above all, combining (\ref{int 012 def}), (\ref{int0 eva}), (\ref{int2 est}), (\ref{int1 est}) and taking $y=N^{1/2}$ we obtain
\begin{equation}\label{M11 mean square}
\begin{aligned}
  \int_{T}^{2T}{M_{11}}^2(x,y,N)dx=&\frac{r^2}{6\pi^2}\sum_{\ell_1,\ell_2=0}^{2(r-1)}D_{r,3,\ell_1,\ell_2}\int_{T}^{2T}x^{2r-4/3}(\log x)^{\ell_1+\ell_2}dx\\
  &\qquad+O(T^{2r-1/3+\varepsilon}N^{-1/3})+O(T^{2r-2/3+\varepsilon}N^{2/3}).
\end{aligned}
\end{equation}

\subsubsection{Other terms in (\ref{int M_1 chaifen})}
\ 
\newline\indent From (\ref{k div aver}) and (\ref{M1 chaifen}) we get
\begin{equation*}
\begin{aligned}
  &\qquad\int_{T}^{2T}{M_{12}}^2(x,y,N)dx\\
  &\ll\sum_{m_1,\cdots,m_{2r}\leqslant y}\frac{|f_{r,3}(\mathbf{M}_1)f_{r,3}(\mathbf{M}_2)|}{D_1D_2}
  \int_{T}^{2T}x^{2r-2+\varepsilon}\left|\delta_{32}\left(\frac{x}{m_r},N\right)\delta_{32}\left(\frac{x}{m_{2r}},N\right)\right|dx\\
  &\ll T^{2r-2+\varepsilon}\sum_{m_1,\cdots,m_{2r}\leqslant y}\frac{|f_{r,3}(\mathbf{M}_1)f_{r,3}(\mathbf{M}_2)|}{D_1D_2}
  \int_{T}^{2T}\left|\delta_{32}\left(\frac{x}{m_r},N\right)\delta_{32}\left(\frac{x}{m_{2r}},N\right)\right|dx.
\end{aligned}
\end{equation*}
Suppose that $N\ll (T/y)^{2/3}$, using the Cauchy-Schwarz's inequality and Lemma \ref{delta32 mean square} we have
\begin{equation*}
\begin{aligned}
  &\qquad\int_{T}^{2T}\left|\delta_{32}\left(\frac{x}{m_r},N\right)\delta_{32}\left(\frac{x}{m_{2r}},N\right)\right|dx\\
  &\ll\left(\int_{T}^{2T}{\delta_{32}}^2\left(\frac{x}{m_r},N\right)dx\right)^{1/2}\left(\int_{T}^{2T}{\delta_{32}}^2\left(\frac{x}{m_{2r}},N\right)dx\right)^{1/2}\\
  &\ll\left(m_r\int_{\frac{T}{m_r}}^{\frac{2T}{m_r}}{\delta_{32}}^2\left(u,N\right)du\right)^{1/2}\left(m_{2r}
  \int_{\frac{T}{m_{2r}}}^{\frac{2T}{m_{2r}}}{\delta_{32}}^2\left(u,N\right)du\right)^{1/2}\\
  &\ll\left(\frac{T^{5/3+\varepsilon}N^{-1/3}}{{m_r}^{2/3}}+\frac{T^{14/9+\varepsilon}}{m_r^{5/9}}\right)^{1/2}
  \left(\frac{T^{5/3+\varepsilon}N^{-1/3}}{{m_{2r}}^{2/3}}+\frac{T^{14/9+\varepsilon}}{m_{2r}^{5/9}}\right)^{1/2}\\
  &\ll\frac{T^{5/3+\varepsilon}N^{-1/3}}{(m_rm_{2r})^{1/3}}+\frac{T^{29/18+\varepsilon}N^{-1/6}}{{m_r}^{5/18}{m_{2r}}^{1/3}}
  +\frac{T^{29/18+\varepsilon}N^{-1/6}}{{m_r}^{1/3}{m_{2r}}^{5/18}}+\frac{T^{14/9+\varepsilon}}{(m_rm_{2r})^{5/18}}.
\end{aligned}
\end{equation*}
Thus
\begin{equation}\label{M12 mean square}
\begin{aligned}
   &\qquad\int_{T}^{2T}{M_{12}}^2(x,y,N)dx\\
  &\ll T^{2r-2+\varepsilon}\sum_{m_1,\cdots,m_{2r}\leqslant y}\frac{|f_{r,3}(\mathbf{M}_1)f_{r,3}(\mathbf{M}_2)|}{D_1D_2}\\
  &\qquad\times\left(\frac{T^{5/3+\varepsilon}N^{-1/3}}{(m_rm_{2r})^{1/3}}+\frac{T^{29/18+\varepsilon}N^{-1/6}}{{m_r}^{5/18}{m_{2r}}^{1/3}}
  +\frac{T^{29/18+\varepsilon}N^{-1/6}}{{m_r}^{1/3}{m_{2r}}^{5/18}}+\frac{T^{14/9+\varepsilon}}{(m_rm_{2r})^{5/18}}\right)\\
  &\ll T^{2r-1/3+\varepsilon}N^{-1/3}+T^{2r-7/18+\varepsilon}N^{-1/6}+T^{2r-4/9+\varepsilon},
\end{aligned}
\end{equation}
where the convergence of all the series can be obtained by Lemma \ref{Toth1}. 

Since $T^{2r-7/18+\varepsilon}N^{-1/6}=(T^{2r-1/3+\varepsilon}N^{-1/3})^{1/2}(T^{2r-4/9+\varepsilon})^{1/2}$, the term $T^{2r-7/18+\varepsilon}N^{-1/6}$ is superfluous.

It remains to estimate 
\begin{equation*}
  \int_{T}^{2T}M_{11}(x,y,N)M_{12}(x,y,N)dx.
\end{equation*}
Using (\ref{M11 mean square}), (\ref{M12 mean square}) and the Cauchy-Schwarz's inequality, we deduce that 
\begin{equation}\label{M11M12 int}
\begin{aligned}
  &\qquad\int_{T}^{2T}M_{11}(x,y,N)M_{12}(x,y,N)dx\\
  &\ll\left(\int_{T}^{2T}{M_{11}}^2(x,y,N)dx\right)^{1/2}\left(\int_{T}^{2T}{M_{12}}^2(x,y,N)dx\right)^{1/2}\\
  &\ll T^{r-1/6+\varepsilon}(T^{r-1/6+\varepsilon}N^{-1/6}+T^{r-2/9+\varepsilon})\\
  &=T^{2r-1/3+\varepsilon}N^{-1/6}+T^{2r-7/18+\varepsilon}.
\end{aligned}
\end{equation}
Combining (\ref{int M_1 chaifen}), (\ref{M11 mean square}), (\ref{M12 mean square}) and (\ref{M11M12 int}) we conclude that
\begin{equation}\label{M1 mean square}
\begin{aligned}
  \int_{T}^{2T}{M_1}^2(x,y)dx=&\frac{r^2}{6\pi^2}\sum_{\ell_1,\ell_2=0}^{2(r-1)}D_{r,3,\ell_1,\ell_2}\int_{T}^{2T}x^{2r-4/3}(\log x)^{\ell_1+\ell_2}dx\\
  &\qquad+O(T^{2r-1/3+\varepsilon}N^{-1/6}+T^{2r-2/3+\varepsilon}N^{2/3}+T^{2r-7/18+\varepsilon}).
\end{aligned}
\end{equation}

\subsection{Conclusion}
 Taking $y=N^{1/2}$, using (\ref{M_2 mean square}), (\ref{M1 mean square}) and the Cauchy-Schwarz's inequality we easily have 
\begin{equation}\label{M1M2 int}
\begin{aligned}
  \int_{T}^{2T}M_1(x,y)M_2(x,y)dx&\ll\left(\int_{T}^{2T}{M_1}^2(x,y)dx\right)^{1/2}\left(\int_{T}^{2T}{M_2}^2(x,y)dx\right)^{1/2}\\
  &\ll T^{r-1/6+\varepsilon}\times T^{r-1/6+\varepsilon}N^{-1/6}\\
  &=T^{2r-1/3+\varepsilon}N^{-1/6}.
\end{aligned} 
\end{equation}
And similarly we obtain
\begin{equation}\label{M1M3 int}
\begin{aligned}
  \int_{T}^{2T}M_1(x,y)M_3(x,y)dx&\ll T^{2r-1/3+\varepsilon}N^{-1/6}.
\end{aligned} 
\end{equation}
Above all, choosing $N=T^{1/3}$ which indeed satisfies $N\ll(T/y)^{2/3}=T^{1/2}$,
\begin{equation*}
  \int_{T}^{2T}{\Delta_{r,3}}^2(x)dx=\frac{r^2}{6\pi^2}\sum_{\ell_1,\ell_2=0}^{2(r-1)}D_{r,3,\ell_1,\ell_2}\int_{T}^{2T}x^{2r-4/3}(\log x)^{\ell_1+\ell_2}dx+O(T^{2r-7/18+\varepsilon})
\end{equation*}
holds from (\ref{Deltark exp}), (\ref{Delta* sq exp}), (\ref{M_3 mean square}), (\ref{M_2 mean square}), (\ref{M1 mean square}), (\ref{M1M2 int}) and (\ref{M1M3 int}).
Then replacing $T$ by $T/2$, $T/2^2$ and so on, and adding up all the results, we obtain
\begin{equation*}
\begin{aligned}
  \int_{1}^{T}{\Delta_{r,3}}^2(x)dx&=\frac{r^2}{6\pi^2}\sum_{\ell_1,\ell_2=0}^{2(r-1)}D_{r,3,\ell_1,\ell_2}\int_{1}^{T}x^{2r-4/3}(\log x)^{\ell_1+\ell_2}dx+O(T^{2r-7/18+\varepsilon})\\
  &=\frac{r^2}{6\pi^2}T^{2r-1/3}L_{4r-4}(\log T)+O(T^{2r-7/18+\varepsilon}),
\end{aligned}
\end{equation*}
where we use integration by part several times to get $L_{4r-4}(u)$ is a polynomial in $u$ of degree $4r-4$ denoted by
\begin{equation*}
\begin{aligned}
  L_{4r-4}(u)=\sum_{\ell_1,\ell_2=0}^{2(r-1)}D_{r,3,\ell_1,\ell_2}\sum_{t=0}^{\ell_1+\ell_2}\frac{(-1)^t(\ell_1+\ell_2)!}{(2r-\frac{1}{3})^{t+1}(\ell_1+\ell_2-t)!}u^{\ell_1+\ell_2-t}.
\end{aligned}
\end{equation*}
Hence we have completed the proof of the Theorem.

\section{Proof of Theorem \ref{third moment}}

By (\ref{Deltark exp}) we obtain
\begin{equation*}
\begin{aligned}
  \int_{T}^{2T}|\Delta_{r,3}(x)|^3dx&\ll\int_{T}^{2T}|{\Delta_{r,3}}^{*}(x)|^3dx\\
  &\ll\int_{T}^{2T}\left(\sum_{m_1,\cdots,m_r\leqslant 2T}|f_{r,3}(m_1,\cdots,m_r)|\prod_{j=1}^{r-1}M_3\left(\frac{x}{m_j}\right)\left|\Delta_3\left(\frac{x}{m_r}\right)\right|\right)^3dx\\
  &\ll T^{3r-3+\varepsilon}\int_{T}^{2T}\left(\sum_{m_1,\cdots,m_r\leqslant 2T}
  \frac{|f_{r,3}(m_1,\cdots,m_r)|}{m_1\cdots m_{r-1}}\left|\Delta_3\left(\frac{x}{m_r}\right)\right|\right)^3dx.
\end{aligned}
\end{equation*}

Using the H\"{o}lder's inequality we have
\begin{equation*}
\begin{aligned}
  &\left(\sum_{m_1,\cdots,m_r\leqslant 2T}\frac{|f_{r,3}(m_1,\cdots,m_r)|}{m_1\cdots m_{r-1}}\left|\Delta_3\left(\frac{x}{m_r}\right)\right|\right)^3\\
  \ll&\left(\sum_{m_1,\cdots,m_r\leqslant 2T}\left(\frac{|f_{r,3}(m_1,\cdots,m_r)|}{m_1\cdots m_{r-1}}\right)^{3/2}\right)^{2}
  \left(\sum_{m_1,\cdots,m_r\leqslant 2T}\left|\Delta_3\left(\frac{x}{m_r}\right)\right|^3\right),
\end{aligned}
\end{equation*}
so we obtain
\begin{equation*}
\begin{aligned}
  &\int_{T}^{2T}|\Delta_{r,3}(x)|^3dx\\
  \ll&T^{3r-3+\varepsilon}\int_{T}^{2T}\left(\sum_{m_1,\cdots,m_r\leqslant 2T}\left(\frac{|f_{r,3}(m_1,\cdots,m_r)|}{m_1\cdots m_{r-1}}\right)^{3/2}\right)^{2}
  \left(\sum_{m_1,\cdots,m_r\leqslant 2T}\left|\Delta_3\left(\frac{x}{m_r}\right)\right|^3\right)dx.
\end{aligned}
\end{equation*}
Changing the order of integration and summation we have
\begin{equation*}
\begin{aligned}
  \int_{T}^{2T}|\Delta_{r,3}(x)|^3dx&\ll T^{3r-3+\varepsilon}\left(\sum_{m_1,\cdots,m_r\leqslant 2T}\frac{|f_{r,3}(m_1,\cdots,m_r)|}{m_1\cdots m_{r-1}}\right)^3\int_{T}^{2T}\left|\Delta_3\left(\frac{x}{m_r}\right)\right|^3dx\\
  &\ll T^{3r-3+\varepsilon}\left(\sum_{m_1,\cdots,m_r\leqslant 2T}\frac{|f_{r,3}(m_1,\cdots,m_r)|}{m_1\cdots m_{r-1}}\right)^3
  \left(m_r\int_{\frac{T}{m_r}}^{\frac{2T}{m_r}}\left|\Delta_3\left(u\right)\right|^3du\right)\\
  &\ll T^{3r-1+\varepsilon}\left(\sum_{m_1,\cdots,m_r\leqslant 2T}\frac{|f_{r,3}(m_1,\cdots,m_r)|}{m_1\cdots m_{r-1}{m_r}^{1/3}}\right)^3\\
  &\ll T^{3r-1+\varepsilon},
\end{aligned}
\end{equation*}
where we use (\ref{H-B}) and Lemma \ref{Toth1}. And Theorem \ref{third moment} holds by replacing $T$ by $T/2$, $T/2^2$ and so on, and adding up all the results.

\section{Proof of Theorem \ref{first moment}}

By (\ref{Deltark exp}) we have
\begin{equation*}
\begin{aligned}
  \int_{1}^{T}\Delta_{r,3}(x)dx&=\int_{1}^{T}{\Delta_{r,3}}^{*}(x)dx+O(T^{r+\varepsilon})\\
  &=r\int_{1}^{T}\sum_{m_1,\cdots,m_r\leqslant x}f_{r,3}(m_1,\cdots,m_r)\prod_{j=1}^{r-1}M_3\left(\frac{x}{m_j}\right)\Delta_3\left(\frac{x}{m_r}\right)dx+O(T^{r+\varepsilon})\\
  &:=I_1+I_2+O(T^{r+\varepsilon}),
\end{aligned}
\end{equation*}
where
\begin{equation*}
\begin{aligned}
  I_1&=r\int_{1}^{T}\sum_{m_1,\cdots,m_r\leqslant 2T}f_{r,3}(m_1,\cdots,m_r)\prod_{j=1}^{r-1}M_3\left(\frac{x}{m_j}\right)\Delta_3\left(\frac{x}{m_r}\right)dx,\\
  I_2&=r\int_{1}^{T}\left(\sum_{m_1,\cdots,m_r\leqslant 2T}-\sum_{m_1,\cdots,m_r\leqslant x}\right)
  f_{r,3}(m_1,\cdots,m_r)\prod_{j=1}^{r-1}M_3\left(\frac{x}{m_j}\right)\Delta_3\left(\frac{x}{m_r}\right)dx.
\end{aligned}
\end{equation*}

We are going to estimate $I_2$. Since $m_1,\cdots,m_{r-1}$ are symmetric in $I_2$, we have
\begin{equation*}
  I_2\ll I_{21}+I_{22},
\end{equation*}
where
\begin{equation*}
\begin{aligned}
  I_{21}&=\int_{1}^{T}\sum_{\substack{m_1,\cdots,m_r\leqslant 2T\\m_1>x}}|f_{r,3}(m_1,\cdots,m_r)|\prod_{j=1}^{r-1}M_3\left(\frac{x}{m_j}\right)\left|\Delta_3\left(\frac{x}{m_r}\right)\right|dx,\\
  I_{22}&=\int_{1}^{T}\sum_{\substack{m_1,\cdots,m_r\leqslant 2T\\m_r>x}}|f_{r,3}(m_1,\cdots,m_r)|\prod_{j=1}^{r-1}M_3\left(\frac{x}{m_j}\right)\left|\Delta_3\left(\frac{x}{m_r}\right)\right|dx.
\end{aligned}
\end{equation*}
Changing the order of integration and summation, using (\ref{M prod exp}) and $\Delta_3(x/m_r)\ll(x/m_r)^{\alpha_3+\varepsilon}$ we have
\begin{equation*}
\begin{aligned}
  I_{22}&\ll\sum_{\substack{m_1,\cdots,m_r\leqslant 2T\\m_r>x}}\frac{|f_{r,3}(m_1,\cdots,m_r)|}{m_1\cdots m_{r-1}{m_r}^{\alpha_3}}\int_{1}^{T}x^{r-1+\alpha_3+\varepsilon}dx\\
  &\ll T^{r+\alpha_3+\varepsilon}\sum_{\substack{m_1,\cdots,m_r\leqslant 2T\\m_r>x}}
  \frac{|f_{r,3}(m_1,\cdots,m_r)|}{m_1\cdots m_{r-1}{m_r}^{\varepsilon}}\cdot\frac{1}{{m_r}^{\alpha_3-\varepsilon}}\\
  &\ll T^{r+\varepsilon}\sum_{m_1,\cdots,m_r=1}^{\infty}\frac{|f_{r,3}(m_1,\cdots,m_r)|}{m_1\cdots m_{r-1}{m_r}^{\varepsilon}}\\
  &\ll T^{r+\varepsilon},
\end{aligned}
\end{equation*}
and the convergence of the latter series is obtained by Lemma \ref{Toth1}.

For $I_{21}$, similar to $I_{22}$ and using Lemma \ref{Toth1} we have
\begin{equation*}
\begin{aligned}
  I_{21}&\ll\sum_{\substack{m_1,\cdots,m_r\leqslant 2T\\m_1>x}}\frac{|f_{r,3}(m_1,\cdots,m_r)|}{m_1\cdots m_{r-1}{m_r}^{\alpha_3}}\int_{1}^{T}x^{r-1+\alpha_3+\varepsilon}dx\\
  &\ll T^{r+\alpha_3+\varepsilon}\sum_{\substack{m_1,\cdots,m_r\leqslant 2T\\m_1>x}}
  \frac{|f_{r,3}(m_1,\cdots,m_r)|}{{m_1}^{1-\alpha_3}m_2\cdots m_{r-1}{m_r}^{\alpha_3}}\cdot\frac{1}{{m_1}^{\alpha_3-\varepsilon}}\\
  &\ll T^{r+\varepsilon}\sum_{m_1,\cdots,m_r=1}^{\infty}\frac{|f_{r,3}(m_1,\cdots,m_r)|}{{m_1}^{1-\alpha_3}m_2\cdots m_{r-1}{m_r}^{\alpha_3}}\\
  &\ll T^{r+\varepsilon}.
\end{aligned}
\end{equation*}
Then we turns to estimate $I_1$. By (\ref{Deltark exp}) we have
\begin{equation*}
  \prod_{j=1}^{r-1}M_3\left(\frac{x}{m_j}\right)\ll\frac{x^{r-1+\varepsilon}}{m_1\cdots m_{r-1}},
\end{equation*}
and its derivative satisfies
\begin{equation*}
  \left(\prod_{j=1}^{r-1}M_3\left(\frac{x}{m_j}\right)\right)'\ll\frac{x^{r-2+\varepsilon}}{m_1\cdots m_{r-1}}.
\end{equation*}
For a real number $u\geqslant2$, define 
\begin{equation*}
  S(u)=\int_{1}^{u}\Delta_3(x)dx,
\end{equation*}
and we denote $\prod_{j=1}^{r-1}M_3\left(x/m_j\right)$ by $\mathcal{M}(x;m_1,\cdots,m_{r-1})$, then changing the order of integration and summation we have
\begin{equation*}
\begin{aligned}
  I_1=&r\sum_{m_1,\cdots,m_r\leqslant 2T}f_{r,3}(m_1,\cdots,m_r)\int_{1}^{T}\mathcal{M}(x;m_1,\cdots,m_{r-1})\Delta_3\left(\frac{x}{m_r}\right)dx\\
  =&r\sum_{m_1,\cdots,m_r\leqslant 2T}f_{r,3}(m_1,\cdots,m_r)\\
  &\times\left(S(T)\mathcal{M}(T;m_1,\cdots,m_{r-1})-\int_{1}^{T}S(u)\mathcal{M}'(u;m_1,\cdots,m_{r-1})du\right)\\
  \ll&\sum_{m_1,\cdots,m_r\leqslant 2T}|f_{r,3}(m_1,\cdots,m_r)|\left(\frac{T^{r+1/6+\varepsilon}}{m_1\cdots m_{r-1}}+\int_{1}^{T}|S(u)||\mathcal{M}'(u;m_1,\cdots,m_{r-1})|du\right)\\
  \ll&\sum_{m_1,\cdots,m_r\leqslant 2T}|f_{r,3}(m_1,\cdots,m_r)|\left(\frac{T^{r+1/6+\varepsilon}}{m_1\cdots m_{r-1}}+\frac{T^{7/6+\varepsilon}}{m_1\cdots m_{r-1}}\int_{1}^{T}u^{r-2+\varepsilon}du\right)\\
  \ll&T^{r+1/6+\varepsilon}\sum_{m_1,\cdots,m_r\leqslant 2T}\frac{|f_{r,3}(m_1,\cdots,m_r)|}{m_1\cdots m_{r-1}}\cdot\frac{{m_r}^\varepsilon}{{m_r}^\varepsilon}\\
  \ll&T^{r+1/6+\varepsilon}\sum_{m_1,\cdots,m_r=1}^{\infty}\frac{|f_{r,3}(m_1,\cdots,m_r)|}{m_1\cdots m_{r-1}{m_r}^\varepsilon}\\
  \ll&T^{r+1/6+\varepsilon},
\end{aligned}
\end{equation*}
where we use integration by parts, Lemma \ref{Toth1} and Lemma \ref{delta3 first moment}. Hence we complete the proof.
\subsection{Proof of the Corollary}
Suppose $T^{\varepsilon}\ll H\ll T$ is a parameter, by Theorem \ref{mean square 3} we get
\begin{equation}\label{Deltar3 T H}
\begin{aligned}
  &\int_{T}^{T+H}{\Delta_{r,3}}^2(x)dx\\
  =&\frac{r^2}{6\pi^2}\left((T+H)^{2r-1/3}L_{4r-4}(\log(T+H))-T^{2r-1/3}L_{4r-4}(\log T)\right)+O(T^{2r-7/18+\varepsilon})\\
  \asymp&HT^{2r-4/3}
\end{aligned}
\end{equation}
for $H\gg T^{17/18+\varepsilon}$. Trivally we have 
\begin{equation*}
  \int_{T}^{T+H}{\Delta_{r,3}}^2(x)dx\ll|\Delta_{r,3}(T)|\int_{T}^{T+H}|\Delta_{r,3}(x)|dx.
\end{equation*}
By (\ref{alphak}) and theorem \ref{U B for M T} we have $\Delta_{r,3}(T)\ll T^{r-1+43/96+\varepsilon}$. Thus for $H\gg T^{17/18+\varepsilon}$,
\begin{equation}\label{lower bound}
  \int_{T}^{T+H}|\Delta_{r,3}(x)|dx\gg HT^{r-25/32}.
\end{equation}
Taking $H\gg T^{91/96+\varepsilon}$, combining Theorem \ref{first moment} we obtain
\begin{equation*}
  \int_{T}^{T+H}|\Delta_{r,3}(x)|dx\gg\int_{T}^{T+H}\Delta_{r,3}(x)dx.
\end{equation*}
Hence $\Delta_{r,3}(x)$ has at least one sign change in the interval $[T,T+H]$ for $T^{91/96+\varepsilon}\ll H\ll T$, and has at least $T^{5/96-\varepsilon}$ sign changes in $[T,2T]$.

\section*{Acknowledgement}

The authors would like to appreciate the referee for his/her patience in refereeing this paper.
This work is supported by Beijing Natural Science Foundation (Grant No.1242003), and the National Natural Science Foundation of China (Grant No.11971476).

\bibliographystyle{plain}
\bibliography{reference.bib}

\begin{thebibliography}{10}

\bibitem{MR4421948}
Xiaodong Cao, Yoshio Tanigawa, and Wenguang Zhai.
\newblock On hybrid moments of {$\Delta_2(x)$} and {$\Delta_3(x)$}.
\newblock {\em Ramanujan J.}, 58(2):597--631, 2022.

\bibitem{GZ}
Zhen Guo.
\newblock On mean square of the error term of a multivariable divisor function.
\newblock {\em AIMS Mathematics}, 9(10):29197--29219, 2024.

\bibitem{MR1576550}
G.~H. Hardy.
\newblock On {D}irichlet's {D}ivisor {P}roblem.
\newblock {\em Proc. London Math. Soc. (2)}, 15:1--25, 1916.

\bibitem{MR1575368}
G.~H. Hardy and J.~E. Littlewood.
\newblock The {A}pproximate {F}unctional {E}quation in the {T}heory of the
  {Z}eta-{F}unction, with {A}pplications to the {D}ivisor-{P}roblems of
  {D}irichlet and {P}iltz.
\newblock {\em Proc. London Math. Soc. (2)}, 21:39--74, 1923.

\bibitem{MR1159354}
D.~R. Heath-Brown.
\newblock The distribution and moments of the error term in the {D}irichlet
  divisor problem.
\newblock {\em Acta Arith.}, 60(4):389--415, 1992.

\bibitem{MR0792089}
Aleksandar Ivi\'{c}.
\newblock {\em The {R}iemann zeta-function}.
\newblock A Wiley-Interscience Publication. John Wiley \& Sons, Inc., New York,
  1985.
\newblock The theory of the Riemann zeta-function with applications.

\bibitem{MR1215269}
Anatolij~A. Karatsuba.
\newblock {\em Basic analytic number theory}.
\newblock Springer-Verlag, Berlin, russian edition, 1993.

\bibitem{MR0998378}
Ekkehard Kr\"{a}tzel.
\newblock {\em Lattice points}, volume~33 of {\em Mathematics and its
  Applications (East European Series)}.
\newblock Kluwer Academic Publishers Group, Dordrecht, 1988.

\bibitem{MR0046485}
E.~C. Titchmarsh.
\newblock {\em The {T}heory of the {R}iemann {Z}eta-{F}unction}.
\newblock Oxford, at the Clarendon Press,, 1951.

\bibitem{MR0098718}
Kwang-Chang Tong.
\newblock On divisor problems. {II}, {III}.
\newblock {\em Acta Math. Sinica}, 6:139--152, 515--541, 1956.

\bibitem{MR3841555}
L\'{a}szl\'{o} T\'{o}th and Wenguang Zhai.
\newblock On multivariable averages of divisor functions.
\newblock {\em J. Number Theory}, 192:251--269, 2018.

\bibitem{MR2718848}
Kai-Man Tsang.
\newblock Recent progress on the {D}irichlet divisor problem and the mean
  square of the {R}iemann zeta-function.
\newblock {\em Sci. China Math.}, 53(9):2561--2572, 2010.

\bibitem{MR1580627}
Georges Voronoi.
\newblock Sur un probl\`eme du calcul des fonctions asymptotiques.
\newblock {\em J. Reine Angew. Math.}, 126:241--282, 1903.

\end{thebibliography}
\end{document}